\documentclass[a4paper,reqno]{amsart}
\usepackage{amssymb}
\usepackage{enumitem}
\setlist[enumerate]{label=\textnormal{(\roman*)}}
\usepackage{hyperref}

\newtheorem{theorem}{Theorem}[section]
\newtheorem{corollary}[theorem]{Corollary}
\newtheorem{lemma}[theorem]{Lemma}
\newtheorem{proposition}[theorem]{Proposition}
\newtheorem{claim}[theorem]{Claim}
\theoremstyle{definition}

\newtheorem{remark}[theorem]{Remark}

\numberwithin{equation}{section}

\newcommand\RR{\mathbb{R}}

\newcommand\ud{\, \mathrm{d}}
\newcommand\e{\varepsilon}

\newcommand\md{\textnormal{Mod}}
\newcommand\vve{\vec \e\,}

\newcommand\vvep{\vec \e_\perp}

\newcommand\TD{T_\delta}
\newcommand\TS{T_*}
\newcommand\Tu{t_1}
\newcommand\Td{t_2}
\newcommand\ENE{{H^1\times L^2}}
\newcommand\LL{\mathcal L}
\newcommand\BB{\mathcal G}

\DeclareMathOperator{\spn}{span}
\DeclareMathOperator{\Id}{Id}
\DeclareMathOperator{\card}{card}

\begin{document}

\parindent=0pt

\title[Asymptotics of 1D damped NLKG equation]
{Long-time asymptotics of the one-dimensional damped nonlinear Klein-Gordon equation}

\author[R.~C\^ote]{Rapha\"el C\^ote}
\address{IRMA UMR 7501, Universit\'e de Strasbourg, CNRS, F-67000 Strasbourg, France}
\email{cote@math.unistra.fr}

\author[Y.~Martel]{Yvan Martel}
\address{CMLS, \'Ecole polytechnique, CNRS, Institut Polytechnique de Paris, 91128 Palaiseau Cedex, France}
\email{yvan.martel@polytechnique.edu}

\author[X.~Yuan]{Xu Yuan}
\address{CMLS, \'Ecole polytechnique, CNRS, Institut Polytechnique de Paris, 91128 Palaiseau Cedex, France}
\email{xu.yuan@polytechnique.edu}

\thanks{Y.~M. and X.~Y. thank IRMA, Universit\'e de Strasbourg, for its hospitality.}


\begin{abstract}
For the one-dimensional nonlinear damped Klein-Gordon equation 
\[ \partial_{t}^{2}u+2\alpha\partial_{t}u-\partial_{x}^{2}u+u-|u|^{p-1}u=0 \quad \mbox{on $\RR\times\RR$,}\]
 with $\alpha>0$ and $p>2$, we prove that any global finite energy solution either converges to $0$ or behaves asymptotically as $t\to \infty$ as the sum of $K\geq 1$ decoupled solitary waves. In the multi-soliton case $K\geq 2$, the solitary waves have alternate signs and their distances are of order $\log t$.
\end{abstract}

\maketitle

\section{Introduction}

\subsection{Main results}
We consider the one-dimensional nonlinear focusing damped Klein-Gordon equation
\begin{equation}\label{nlkg}
\partial_{t}^{2}u+2\alpha\partial_{t}u-\partial_{x}^{2}u+u-f(u)=0, \quad (t,x)\in \RR\times\RR,\quad 
f(u)=|u|^{p-1}u,
\end{equation}
with $\alpha>0$ and $p>2$.
It follows from standard arguments that the Cauchy problem for \eqref{nlkg} is locally
well-posed in the energy space $H^1\times L^2$ (see references in~\S\ref{S2.1}). 
Moreover, the existence of solutions blowing up in finite time is well-known, \cite{BRS}.
Denote $F(u)=\frac{1}{p+1}|u|^{p+1}$. Defining the energy of a solution $\vec{u}=(u,\partial_{t}u)$ by
\begin{equation}\label{energy}
E(\vec{u})=\frac{1}{2} \int \big\{(\partial_{t} u)^{2}+(\partial_{x}u)^{2}+u^{2}-2F(u) \big\} \ud x,
\end{equation}
it holds formally
\begin{equation}\label{eq:energy}
E(\vec u(t_2))-E(\vec u(t_1)) = -2 \alpha \int_{t_1}^{t_2} \|\partial_t u(t)\|_{L^2}^2 \ud t.
\end{equation}
It is also well-known that up to sign and translation, the only stationary solution of \eqref{nlkg} is the solitary wave $(Q,0)$, where $Q$ is the explicit ground state
\begin{equation}\label{def:Q}
Q(x)=\left(\frac{p+1}{2\cosh^{2}\big(\frac{p-1}{2}x\big)}\right)^{\frac{1}{p-1}},
\end{equation}
which solves the equation
\begin{equation}\label{eq:Q}
Q''-Q+Q^{p}=0 \quad \mbox{on $\RR$.}
\end{equation}

Remarkable results and techniques developed in \cite{BRS,Cnlkg,F98,PLL1,PLL2} provide information on the long-time asymptotic behavior of global solutions of \eqref{nlkg} and of higher dimensional variants of this model for suitable nonlinearities.
From techniques in~\cite{BRS,Cnlkg}, all global solutions are proved to be
bounded in the energy space (see Theorem~\ref{pr:bound}). From~\cite{F98} and the concentration-compactness principle as stated in~\cite{PLL1,PLL2}, any global solution either converges to zero in the energy space, or decomposes along a subsequence of time into a sum of decoupled ground states (see Theorem~\ref{th:1} and Remark~\ref{rk:subseq}). 
For space dimension $2\leq N\leq 6$, it is proved in~\cite{BRS} that any global radially symmetric solution converges either to $0$ or to a single solitary wave as $t\to \infty$.
In \cite{CMYZ}, in space dimension $1\le N\le 5$ and for energy subcritical nonlinearities, global solutions of the damped Klein-Gordon equation containing two solitary waves are described and classified.

The objective of the present article is to complement those works by describing precisely the asymptotic behavior as $t\to \infty$ of any global solution of~\eqref{nlkg}.
The choice of considering the one-dimensional model is discussed in Remark~\ref{rk:dim}.

\begin{theorem}\label{th:desc}
For any global solution $\vec{u}\in C\big([0,\infty),H^{1}\times L^{2}\big)$ of~\eqref{nlkg}, one of the following three scenarios occurs:
\begin{description}
\item[Vanishing] $\vec{u}(t)$ converges exponentially to $0$ in $H^{1}\times L^{2}$ as $t\to \infty$.
\item[Single soliton] There exist $\sigma=\pm 1$, $\ell \in \RR$ such that $\vec{u}(t)$ converges exponentially to $(\sigma Q(\cdot-\ell),0)$ in $H^{1}\times L^{2}$ as $t\to \infty$.
\item[Multi-soliton] There exist $K\geq 2$, $\sigma=\pm 1$, $\ell \in \RR$ and
functions $ z_k: [0,\infty)\to \RR$, for all $k=1,\ldots,K$ such that for all $t\in (0,\infty)$,
\begin{equation}\label{eq:th:2}
\left\|u(t)-\sigma \sum_{k=1}^K(-1)^kQ(\cdot- z_k(t))\right\|_{H^1}
+\|\partial_{t}u(t)\|_{L^{2}}\lesssim t^{-1},
\end{equation}
and for any $1<\theta<\min(p-1,\frac{5}{4})$, $k=1,\dots, K$, 
\begin{equation} \label{eq:th:z}
z_k(t) = \left( k - \frac{K+1}{2} \right) \log t + \tau_k + \ell + O(t^{-\theta+1}),
\end{equation}
as $t\to \infty$, where $\tau_k$ are the constants uniquely defined by 
\begin{equation} \label{def:tau_k}
\sum_{k=1}^K \tau_k = 0, \quad e^{-(\tau_{k+1} - \tau_{k})} = \frac{2\alpha}{\kappa} \gamma_k \quad \text{where} \quad \gamma_k = \frac{k (K-k)}{2},
\end{equation}
and $\kappa>0$ defined in~\eqref{def:kappa} only depends on $p$.
\end{description}
\end{theorem}
\begin{remark}
The parameter $\ell$ is related to the translation invariance.
In the vanishing and the single soliton cases, the damping leads to exponential convergence as $t\to \infty$.
In the multi-soliton case, due to the nonlinear interactions between the solitary waves, the asymptotic behavior~\eqref{eq:th:z} of the centers of mass $z_k$ is related to the following nonlinear differential system
\begin{equation}\label{eq:uedo}
\left\{ \begin{aligned}
&\dot{y}_{1}=-\frac{\kappa}{2\alpha} e^{-(y_2-y_1)},\\
& \dot y_k = \frac{\kappa}{2\alpha} \left( e^{-(y_k - y_{k-1})} - e^{-(y_{k+1} - y_{k})} \right),\quad \mbox{for}\ k=2,\cdots,K-1,\\
&\dot{y}_{K}=\frac{\kappa}{2\alpha}e^{-(y_{K}-y_{K-1})}.
\end{aligned}\right.
\end{equation}
This system is studied in~\cite{CZ,MZ} in the context of blowup solutions of the nonlinear wave equation. The nonlinear interactions at short distances between the solitary waves also yield the slower rate of convervence $t^{-1}$ in \eqref{eq:th:2}.
See~\cite{Jkdv, TVNkdv,TVN} for other examples of strong interactions leading to $\log t$ distant solitary waves.
\end{remark}

Theorem~\ref{th:desc} is a version of the soliton resolution for global solutions of the one-dimensional damped nonlinear Klein-Gordon equation~\eqref{nlkg},
with convergence for the whole  sequence of time and a description of the parameters of the solitary waves.
The case of the damped Klein-Gordon equation on a bounded domain for suitable nonlinearities is addressed in \cite[Theorem~9.5.3]{CH}; see also references therein.

We refer to~\cite{C13,DKM,DJKM}
for results related to the soliton resolution conjecture for the challenging case of the undamped energy critical wave type equation.

Obviously, $\vec{u}=(0,0)$ and $\vec{u}=(Q,0)$ are examples of the first two scenarios of Theorem~\ref{th:desc}.
Our second result gives examples of the third scenario for any $K\geq 2$.

\begin{theorem} \label{th:2}
For any $K\geq 2$, $\sigma = \pm 1$ and $\ell \in \RR$, there exist global solutions of~\eqref{nlkg} satisfying
~\eqref{eq:th:2} and \eqref{eq:th:z}.
\end{theorem}
 
\subsection{Notation and basic results on the solitary wave}
We denote $\langle \cdot, \cdot \rangle$ the $L^2$ scalar product for real-valued functions $u_i$ or vector-valued functions $\vec u_i = (u_i,v_i)$ ($i=1,2$)
\[ \langle u_1, u_2 \rangle := \int u_1(x) u_2(x) \ud x, \quad \langle \vec u_1, \vec u_2 \rangle := \int u_1(x) u_2(x) \ud x +\int v_1(x) v_2(x) \ud x. \]

We see from the explicit expression of~$Q$ in \eqref{def:Q} that, as $x\to \infty$,
\begin{equation}\label{asym:Q}
Q(x)=c_{Q}e^{-x}+O(e^{-2x}),\quad Q'(x)=-c_{Q}e^{-x}+O(e^{-2x})
\end{equation}
where $c_{Q}=(2p+2)^{\frac{1}{p-1}}$.
Note that by \eqref{eq:Q}, it holds $\int (\partial_x Q)^2+Q^2-Qf(Q)=0$ and so
\begin{equation}\label{eq:Epos}
E(Q,0)=\left(\frac 12 - \frac 1{p+1}\right) \int Q^{p+1} >0.
\end{equation}

Let
\[
 \LL = -\partial_x^2+1-p Q^{p-1} ,\quad
\langle \LL \e,\e\rangle = \int \big\{|\partial_x \e|^2+\e^{2} - p Q^{p-1} \e^2\big\}\ud x.
\]
We recall some standard properties of the operator $\LL $ (see \emph{e.g.}~\cite[Lemma 1]{CMkg}).
\begin{lemma}\label{le:L}
The following properties hold.
\begin{enumerate}
\item \emph{Spectral properties.} The unbounded operator $\LL $ on $L^2$ with domain $H^2$ is self-adjoint, its continuous spectrum is $[1,\infty)$, its kernel is $\spn\{Q'\}$ and it has a unique negative eigenvalue $-\nu_{0}^{2}$, with corresponding smooth normalized eigenfunction $Y$ $(\|Y\|_{L^2}=1)$.
Moreover, on $\RR$,
\begin{equation*}
 |Y^{(n)}(x) |\lesssim e^{-\sqrt{1+\nu_{0}^{2}}\left|x\right|}\quad \text{for any } n\in \mathbb{N}.
\end{equation*}
\item \emph{Coercivity property.} There exists $c>0$ such that, for all $\e\in H^{1}$,
\begin{equation*}
\langle \LL \e,\e\rangle\ge c
\|\e\|_{H^{1}}^{2}-c^{-1}
\left(\langle \e,Y\rangle^{2} + \langle \e,Q'\rangle^{2}\right).
\end{equation*}
\end{enumerate}
\end{lemma}
Recall that the unique negative eigenvalue of $\LL $ is related to an instability of the solitary wave
for the equation~\eqref{nlkg}, described by the following functions:
\begin{gather}
\nu^\pm = - \alpha \pm \sqrt{\alpha^2+\nu_0^2},\quad 
\vec Y^\pm = \begin{pmatrix}
Y\\ \nu^\pm Y
\end{pmatrix},\label{eq:Y}\\
\zeta^\pm = \alpha \pm \sqrt{\alpha^2+\nu_0^2},\quad
\vec Z^\pm = \begin{pmatrix}
\zeta^\pm Y\\ Y
\end{pmatrix}.\label{eq:Z}
\end{gather}

\section{General properties of finite energy solutions}
In this section, we gather some known material on finite energy solutions of~\eqref{nlkg}.
 We repeat some proofs for the sake of completeness.

\subsection{Cauchy problem in the energy space}\label{S2.1}
It is well-known (see for instance \cite[Chapter 9.5]{CH}) that the linear problem
\[
\partial_{t}^{2} u + 2 \alpha \partial_t u - \partial_x^2 u + u = 0 \quad (t,x)\in \RR\times \RR
\]
generates a strongly continuous semigroup of contractions $(S_\alpha(t))_{t\geq 0}$ 
in $H^1\times L^2$ or $L^2\times H^{-1}$ satisfying, for some
$C\geq 1$, $\gamma>0$,
\begin{equation}\label{eq:expoSG}
\|S_\alpha(t)\|_{\mathcal L(H^1\times L^2)} \leq C e^{-\gamma t},\quad
\|S_\alpha(t)\|_{\mathcal L(L^2\times {H^{-1}})} \leq C e^{-\gamma t},
\end{equation}
for all $t\geq 0$.
Recall also that the map $u\mapsto f(u)$ is Lipschitz continuous from bounded sets of $H^1$ to $L^2$.
In particular, the standard theory of semilinear evolution equations
(see for instance \cite[Chapter 4.3]{CH} or \cite{Pazy}) yields the following result.

\begin{proposition}
For any initial data
$(u_0,v_0)\in H^1\times L^2$, there exists a unique maximal solution
\[
\vec{u}=(u,\partial_{t}u) \in C([0,T_{\max}),H^1\times L^2)\cap C^1([0,T_{\max}),L^2\times H^{-1})\]
 of~\eqref{nlkg} satisfying $\vec{u}(0)=(u_0,v_0)$. If the maximal time of existence $T_{\max}$ is finite, then
 $\lim_{t\uparrow T_{\max}} \|\vec u(t)\|_\ENE=\infty$.
 
Moreover, the map $T_{\max}:(u_0,v_0)\in H^1\times L^2\mapsto (0,\infty]$ is lower semicontinuous, and
if $(u_{0,n},v_{0,n})\to_{n\to \infty} (u_0,v_0)$ in $H^1\times L^2$ then, 
for any $0<T<T_{\max}$, 
\[
(u_n,\partial_t u_n)\to (u,\partial_t u) \quad \mbox{in $C([0,T],H^1\times L^2)$,}
\]
where $(u_n,\partial_t u_n)$ is the solution of \eqref{nlkg} corresponding to $(u_{0,n},v_{0,n})$.
\end{proposition}
In this paper, we systematically work in the framework of such maximal finite energy solutions, for which it is standard to check that the relation~\eqref{eq:energy} holds.
We call \emph{global solution} a solution for which $T_{\max}=\infty$. We do not consider solutions of \eqref{nlkg} backwards in time (\emph{i.e.} for negative values of $t$).

\subsection{Bound on global solutions}
Gathering the arguments of \cite{Cnlkg} and \cite[Proof of Lemma 2.7]{BRS}, we recall the following
bound on global solutions of \eqref{nlkg}. Note that~\cite{Cnlkg} is devoted to the undamped Klein-Gordon equation, but as suggested in the Introduction of \cite{BRS}, the proof extends to the damped case.
\begin{theorem}[\cite{BRS,Cnlkg}]\label{pr:bound}
Any global solution of~\eqref{nlkg} is bounded in $H^1\times L^2$.
\end{theorem}
\begin{proof}
Let $\vec u$ be a global solution of \eqref{nlkg}.
Together with the energy functional $E(t):=E(\vec u(t))$ defined in \eqref{energy} and satisfying
\eqref{eq:energy}, we will use the following quantities
\begin{align*}
M(t)& := \frac 12 \|u(t)\|_{L^2}^2 + \alpha \int_0^t \|u(s)\|_{L^2}^2 \ud s,\\
W(t)& := \frac 12 \left(\|\partial_tu(t)\|_{L^2}^2 + \|\partial_x u(t)\|_{L^2}^2 +\|u(t)\|_{L^2}^2\right).
\end{align*}
By direct computations using \eqref{nlkg} and~\eqref{energy}, we check the following relations
\begin{align}
M'(t) &= \int u (t) \partial_t u (t) \ud x + \alpha \|u(t)\|_{L^2}^2 \label{eq:dM}\\
& =\int u (t) \partial_t u (t) \ud x 
 + 2\alpha\int_0^t \int u (s) \partial_t u (s) \ud x \ud s+ \alpha \|u(0)\|_{L^2}^2,\label{eq:dMc}
\\
M''(t)&= \frac{p+3}2\|\partial_tu(t)\|_{L^2}^2 + \frac{p-1}2\left(\|\partial_x u(t)\|_{L^2}^2 
+\|u(t)\|_{L^2}^2\right) - (p+1) E(t),\label{eq:d2M}\\
W'(t) & = -2\alpha\|\partial_t u(t)\|_{L^2}^2 + \int f(u(t))\partial_t u(t) \ud x.\label{eq:dW}
\end{align}
In particular, from \eqref{eq:dM} and the Cauchy-Schwarz inequality, 
\begin{equation}\label{eq:dMW}
|M'(t)| \leq (1+2\alpha) W(t).
\end{equation}
Moreover, by~\eqref{eq:energy} and~\eqref{eq:d2M},
\begin{equation}\label{eq:d2MW}
M''(t) \geq (p-1) W(t) - (p+1) E(0).
\end{equation}
The proof of the global bound now proceeds in three steps.

\emph{Step 1.} We prove that
\begin{equation}\label{eq:liminf}
\liminf_{t\to \infty} M'(t) <\infty.
\end{equation}
Proof of \eqref{eq:liminf}.
We argue by contradiction, proving that $\lim_\infty M'=\infty$ implies the following inequality,
for all $t$ large enough, 
\begin{equation}\label{eq:blowup}
 (1+\epsilon) [M'(t)]^2 < M''(t) M(t) \quad \mbox{where}\quad \epsilon>0.
\end{equation}
Then, we reach a contradiction by a standard argument.
Indeed, remark that \eqref{eq:blowup} implies $\frac{\ud^2}{\ud t^2} [ M^{-\epsilon}(t)] <0$,
and $\lim_\infty M'=\infty$ also implies $\lim_\infty M^{-\epsilon}=0$. Thus, there exists $t_1>0$
such that $\frac{\ud}{\ud t} [ M^{-\epsilon}(t_1)] <0$, and for all $t\geq t_1$,
\[
0\leq M^{-\epsilon}(t) \leq M^{-\epsilon}(t_1) + (t-t_1) \frac{\ud}{\ud t} [ M^{-\epsilon}(t_1)],
\]
which is absurd for $t\geq t_1$ large enough.

Thus, we only need to prove \eqref{eq:blowup} assuming $\lim_\infty M'=\infty$.
On the one hand, by \eqref{eq:dMc} and the Cauchy-Schwarz inequality, it holds
\[
|M'|\leq \|u\|_{L^2}\|\partial_t u\|_{L^2} + 2\alpha \left(\int_0^t \|u(s)\|_{L^2}^2 \ud s\right)^{\frac 12}
\left( \int_0^t\|\partial_t u(s)\|_{L^2}^2 \ud s\right)^{\frac 12} + \alpha \|u(0)\|_{L^2}^2.
\]
Let $\epsilon>0$ to be chosen later, we estimate
\begin{align*}
|M'|^2
&\leq (1+\epsilon) \left[ \|u\|_{L^2}\|\partial_t u\|_{L^2} + 2\alpha \left(\int_0^t \|u(s)\|_{L^2}^2 \ud s\right)^{\frac 12}
\left(\int_0^t\|\partial_t u(s)\|_{L^2}^2 \ud s\right)^{\frac 12} \right]^2\\
& \quad + \left(1+\frac1\epsilon\right)\alpha^2 \|u(0)\|_{L^2}^4.
\end{align*}
Thus, 
\begin{align*}
|M'|^2
&\leq (1+\epsilon) \left[\frac 12 \|u\|_{L^2}^2 + \alpha \int_0^t \|u(s)\|_{L^2}^2 \ud s
 \right]
\left[ 2\|\partial_t u\|_{L^2}^2 + 4 \alpha \int_0^t\|\partial_t u(s)\|_{L^2}^2\ud s\right]\\
& \quad + \left(1+\frac1\epsilon\right)\alpha^2 \|u(0)\|_{L^2}^4\\
&\leq (1+\epsilon) M
\left[ 2\|\partial_t u\|_{L^2}^2 + 4 \alpha \int_0^t\|\partial_t u(s)\|_{L^2}^2 \ud s\right]
 + \left(1+\frac1\epsilon\right)\alpha^2 \|u(0)\|_{L^2}^4.
\end{align*}
On the other hand, by \eqref{eq:energy} and~\eqref{eq:d2M},
\begin{align*}
M'' 
&= 2 \|\partial_t u\|_{L^2}^2 + (p-1) W + 2\alpha (p+1) \int_0^t \|\partial_t u(s)\|_{L^2}^2 \ud s 
 - (p+1) E(0)\\
 &\geq (1+\epsilon)^3 \left[ 2 \|\partial_t u\|_{L^2}^2 + 4\alpha\int_0^t \|\partial_t u(s)\|_{L^2}^2 \ud s\right]
 +\frac{p-1}2 W(t)
 - (p+1) E(0),
\end{align*}
by fixing any $\epsilon$ such that
\[
0<\epsilon < \left(\frac{p+7}{8}\right)^{\frac 13} -1.
\]
 In particular, since $\lim_\infty W =\infty$ by \eqref{eq:dMW} and the 
assumption $\lim_\infty M'=\infty$, we have for $t$ large enough,
\[
M'' \geq (1+\epsilon)^3 \left[ 2 \|\partial_t u\|_{L^2}^2 + 4\alpha\int_0^t \|\partial_t u(s)\|^2 \ud s\right].
\]
Thus,
\[
(1+\epsilon)^2 |M'|^2 \leq MM''+\left(1+\frac1\epsilon\right)\alpha^2 \|u(0)\|_{L^2}^4,\]
and using again $\lim_\infty M'=\infty$ we obtain \eqref{eq:blowup} for any $t$ large enough.

\emph{Step 2.} We prove that
\begin{equation}\label{eq:limsup}
\sup_{t\in [0,\infty)} | M'(t) |<\infty.
\end{equation}
Proof of \eqref{eq:limsup}. 
Combining~\eqref{eq:dMW} and~\eqref{eq:d2MW}, we obtain
\begin{equation*}
M''(t)\geq \frac{p-1}{1+2\alpha} |M'(t)| - (p+1) E(0).
\end{equation*}
Let
\[
H(t) = \frac{p-1}{1+2\alpha} M'(t) - (p+1) E(0).
\]
Then, $H'(t) = \frac{p-1}{1+2\alpha} M''(t) \geq \frac{p-1}{1+2\alpha} H(t)$.
If there exists $t\geq 0$ such that $H(t)>0$, then $\lim_{\infty} H = \infty$,
contradicting \eqref{eq:liminf}. It follows that for all $t\geq 0$,
\[ M'(t) \leq \frac{1+2\alpha}{p-1}(p+1) E(0).\]
Similarly, let
\[
K(t) = -\frac{p-1}{1+2\alpha} M'(t)+ (p+1) E(0).
\]
Then, $K'(t) = -\frac{p-1}{1+2\alpha} M''(t) \leq -\frac{p-1}{1+2\alpha} K(t)$.
It follows that $K(t) \leq e^{-\frac{p-1}{1+2\alpha}t} K(0),$ for all $t\geq 0$.
Thus, 
\[
M'(t) \geq - \frac{1+2\alpha}{p-1} \left( (p+1) E(0) +|K(0)|\right).
\]
and \eqref{eq:limsup} is proved.

\emph{Step 3.} Last, we prove the global bound
\begin{equation}\label{eq:boundW}
\sup_{t\in [0,\infty)} | W(t) |<\infty.
\end{equation}
Proof of \eqref{eq:boundW}.
We rewrite \eqref{eq:d2MW} as
\begin{equation*}
W(t)
\leq \frac1{p-1} M''(t) + \frac{p+1}{p-1} E(0).
\end{equation*}
Integrating on $(t,t+1)$ and using \eqref{eq:limsup}, we observe that
\begin{equation}\label{eq:intW}
\sup_{t\geq 0} \int_t^{t+1} W(s) \ud s <\infty.
\end{equation}
Moreover, by \eqref{eq:dW},
\[
W'\leq -2\alpha\|\partial_t u\|_{L^2}^2 + \int |u|^p |\partial_t u|
\leq \frac 12 \|\partial_t u\|_{L^2}^2+ \frac 12 \int |u|^{2p}
\leq W+\frac 12 \int |u|^{2p}.
\]
For $t\geq 1$ and $\tau\in (0,1)$, integrating on $(t-\tau,t)$, we find
\begin{align*}
W(t) 
&\leq W(t-\tau)+ \int_{t-\tau}^t W(s) \ud s +\frac 12 \int_{t-\tau}^t \int |u(s)|^{2p} \ud x \ud s\\
&\leq W(t-\tau)+ \int_{t-1}^t W(s) \ud s +\frac 12 \int_{t-1}^t \int |u(s)|^{2p} \ud x \ud s.
\end{align*}
Using the Sobolev inequality (in space-time) for the last term, we obtain, for some constants $C>0$,
\begin{align*}
W(t)
&\leq W(t-\tau) + \int_{t-1}^t W(s) \ud s+ C \|u\|_{H^1((t-1,t)\times \RR)}^{2p}\\
&\leq W(t-\tau) + \int_{t-1}^t W(s) \ud s + C \left(\int_{t-1}^t W(s) \ud s\right)^{p}.
\end{align*}
Integrating in $\tau\in(0,1)$ and using~\eqref{eq:intW}, we find~\eqref{eq:boundW}.
\end{proof}

\subsection{Decomposition of any global solution along a subsequence}
\begin{theorem}[\cite{F98,PLL1,PLL2}]\label{th:1}
Any global solution $\vec{u}$ of~\eqref{nlkg}
\begin{itemize}
\item either converges to $0$, \emph{i.e.} $\lim_{t\to\infty} \|\vec u(t)\|_{H^{1}\times L^2} =0\,;$
\item or is asymptotically a (multi-)solitary wave along a subsequence of time: there exist $K\geq 1$,
 a sequence $t_{n}\to \infty$, a sequence $\left(\xi_{k,n}\right)_{k\in\{1,\ldots,K\}}\in \mathbb{R}^{K}$ and signs
$\sigma_{k}=\pm 1$, for any $k\in \{1,\ldots, K\}$, such that
\begin{equation}\label{eq:decompo}
\lim_{n\to\infty}\bigg\{\Big\|u(t_{n})-\sum_{k=1}^{K}\sigma_{k}Q(\cdot-\xi_{k,n})\Big\|_{H^1}
+\|\partial_{t}u(t_{n})\|_{L^{2}}\bigg\}=0
\end{equation}
and in the case $K\geq 2$,
\begin{equation*}
\lim_{n\to \infty} \xi_{k+1,n}-\xi_{k,n}= \infty\quad \mbox{for any $1\leq k\leq K-1$.}
\end{equation*}
\end{itemize}
\end{theorem}
\begin{remark}
It is clear that if a global solution
$\vec u$ satisfies \eqref{eq:decompo} for two different sequences $(t_n)_{n}$ and $(t_n')_{n}$, then the 
number $K\geq 1$ of solitary waves is the same for both sequences.
Indeed, by monotonicity of the energy \eqref{eq:energy} and \eqref{eq:Epos}, it holds
\begin{equation}\label{eq:energlimit}
\lim_{t\to \infty} E(\vec u(t)) = K\, E(Q,0) >0.
\end{equation}
\end{remark}

\begin{remark}\label{rk:subseq}
The following stronger result holds in the framework of Theorem~\ref{th:1}: for any sequence $(t_n)_n$ with $t_n\to \infty$, the multi-solitary wave behavior~\eqref{eq:decompo} is satisfied for a subsequence of $(t_n)_n$.
This result, valid on any global solution of \eqref{nlkg}, is quite remarkable. However, it does not fully describe the asymptotic behavior of global solutions as $t\to \infty$, which is the objective of Theorem~\ref{th:desc}.
\end{remark}

\begin{remark}\label{rk:dim}
Note that \cite{BRS,Cnlkg,CMYZ,F98,PLL1,PLL2,LZ} also apply to the multi-dimensional case, under suitable restrictions on the exponent $p$ of the nonlinearity, or for radially symmetric solutions.
However, for space dimensions greater than $1$, the existence of bound states solutions of $\Delta w-w+f(w)=0$ other than the ground state $Q$, together with the possibility of involved geometric configurations of solitary waves, complicate the analysis. This is why we restrict to dimension $1$ in the present paper.
\end{remark}

\begin{proof}
Let $\vec u$ be a global solution of \eqref{nlkg}; in particular, by Theorem~\ref{pr:bound},
it is bounded in $H^1\times L^2$. The proof proceeds in two steps.

\emph{Step 1.} We prove that
\begin{equation}\label{eq:Hm1z}
\lim_{t\to \infty} \left\{\|\partial_tu (t)\|_{L^2} + \|\partial_t^2 u (t)\|_{H^{-1}}\right\}= 0.
\end{equation}
With the notation of \S\ref{S2.1}, 
the function $\vec v(t)=(v(t),\partial_t v(t))=(\partial_t u(t),\partial_t^2 u(t))$ satisfies
\[
\vec v(t) = S_\alpha(t) \vec v(0) + \int_0^t S_\alpha (t-s) (0,p|u(s)|^{p-1} v(s)) \ud s.
\]
By the bound in $H^1\times L^2$ and \eqref{eq:energy}, it follows that 
$v\in L^2((0,\infty)\times \RR)$.
Moreover, using estimate~\eqref{eq:expoSG}, $\|\cdot\|_{H^{-1}}\lesssim \|\cdot\|_{L^2}$
and $\|\cdot\|_{L^\infty}\lesssim \|\cdot\|_{H^1}$, we have, for all $t\geq 0$,
\[
\|\vec v(t)\|_{L^2\times H^{-1}}
\lesssim e^{-\gamma t} \|\vec v(0)\|_{L^2\times H^{-1}}
+ \|u\|_{L^\infty([0,\infty),H^1)}^{p-1} \int_0^t e^{-\gamma (t-s)} \| v(s)\|_{L^2} \ud s.
\]
Splitting the integral $\int_0^t = \int_0^{t/2}+\int_{t/2}^t$ in the last term and using the Cauchy-Schwarz inequality
\[
\int_0^t e^{-\gamma (t-s)} \| v(s)\|_{L^2} \ud s
\lesssim e^{-\gamma t/2} \| v\|_{L^2((0,\infty)\times \RR)} 
+ \| v\|_{L^2((t/2,\infty)\times \RR)},
\]
which implies $\lim_{t\to \infty}\|\vec v(t)\|_{L^2\times H^{-1}} =0$ and thus \eqref{eq:Hm1z}.

\emph{Step 2.}
Let $(t_n)_n$ be any sequence such that $t_n\to \infty$ and let $u_n(x)=u(t_n,x)$. Then,
by \eqref{eq:Hm1z} and equation~\eqref{nlkg}, it follows that
\begin{equation*}
\lim_{n\to \infty} \|\partial_x^2 u_n - u_n + |u_n|^{p-1} u_n\|_{H^{-1}}= 0.
\end{equation*}
Moreover, the sequence $(u_n)_n$ is bounded in $H^1$.
Then, the alternative stated in the Theorem follows directly from results from concentration-compactness arguments in \cite[Appendix A]{PLL1} and \cite[Theorem~III.4]{PLL2}. 
In the present framework, we use \cite[Theorem~III.4]{PLL2} in space dimension $1$ and with the constant coefficient elliptic operator $-\partial_x^2+1$, which simplifies the statement.
Observe that in dimension $1$, we enjoy the fact that the only non trivial solutions of 
$\partial_x^2 w - w + f(w)=0$ are $w=\pm Q$, up to space translation.

In the case where $\lim_{n\to\infty} \|\vec u(s_n)\|_{H^{1}\times L^2} =0$, for some sequence of time $(s_n)_n$, 
$s_n\to \infty$, then it follows from \eqref{eq:energy} that $\lim_{t\to \infty} E(\vec u(t))=0$. Thus, by \eqref{eq:Epos} and the previous arguments 
applied to any sequence $(t_n)_n$, with $t_n\to \infty$, there exists a subsequence $(t_{n'})_{n'}$ such that $\lim_{n'\to\infty} \|\vec u(t_{n'})\|_{H^{1}\times L^2} =0$.
This implies that $\lim_{t\to\infty} \|\vec u(t)\|_{H^{1}\times L^2} =0$ as stated in the first part of the alternative.
\end{proof}

\section{Dynamics close to decoupled solitary waves}\label{S:2}
In this Section, we prove general results on solutions of \eqref{nlkg} close to the sum of $K\geq 1$ decoupled solitary waves. For any $k\in \{1,\cdots,K\}$, let $\sigma_{k}=\pm 1$ and let
$t\mapsto (z_{k}(t),\ell_{k}(t) )\in \mathbb{R}^{2}$ be $\mathcal C^1$ functions such that
\begin{equation}\label{on:zk}
\sum_{k=1}^K |\ell_k|\ll 1 \quad \mbox{and, if $K\geq 2$, for any $k=1,\ldots,K-1$,}\quad
z_{k+1}-z_k\gg 1.
\end{equation}
For $k\in\{1,\cdots,K\}$, define
\begin{equation}\label{def:gs}
Q_{k}=\sigma_{k}Q(\cdot-z_{k}),\quad
\vec{Q}_{k}=\begin{pmatrix} Q_{k} \\ -{\ell}_{k}\partial_{x} Q_{k} \end{pmatrix},
\end{equation}
and similarly (see \eqref{eq:Y}-\eqref{eq:Z})
\begin{equation*}
Y_k=\sigma_{k} Y(\cdot-z_{k}),\quad
\vec Y_k^\pm= \sigma_{k} \vec Y^\pm (\cdot-z_{k}),\quad
\vec Z_k^\pm= \sigma_{k} \vec Z^\pm (\cdot-z_{k}).
\end{equation*}
Set
\begin{equation}\label{def:G}
R=\sum_{k=1}^{K}Q_k, \quad 
\vec{R}=\sum_{k=1}^{K}\vec{Q}_k,
\quad 
G=f\left(\sum_{k=1}^{K}Q_{k}\right)-\sum_{k=1}^{K}f\left(Q_{k}\right).
\end{equation}
\subsection{Leading order of the nonlinear interactions}
\begin{lemma}
Assuming \eqref{on:zk}, for any $k,k'\in \{1,\ldots,K\}$, $k'\neq k$, it holds.
\begin{enumerate}
\item\emph{Bounds.} For any $0<m'<m$,
\begin{gather}
\int |Q_{k'} Q_k |^m \lesssim e^{-m' |z_{k'}-z_{k}|},\quad
\int |Q_{k'}| |Q_k|^{1+m} \lesssim e^{-|z_{k'}-z_{k}|}, \label{tech1}\\
\int\bigg|F(R)-\sum_{k=1}^{K}F(Q_k)-\sum_{k\ne k'}f(Q_k)Q_{k'}\bigg|
 \lesssim \sum_{k=1}^{K-1}e^{-\frac 54(z_{k+1}-z_{k})},\label{tech3}\\
 \|G\|_{L^2}\lesssim\sum_{k'\ne k} \|Q_{k'}^{p-1}Q_k\|_{L^2}
\lesssim \sum_{k=1}^{K-1}e^{-(z_{k+1}-z_{k})}.\label{tech2}
\end{gather}
\item\emph{Asymptotics.}
\begin{equation}\label{tech4}
\big|\langle f(Q_{k}),Q_{k'}\rangle-\sigma_{k}\sigma_{k'}c_{1}\kappa e^{-|z_{k}-z_{k'}|}\big|\lesssim e^{-\frac{3}{2}|z_{k}-z_{k'}|}
\end{equation}
where
\begin{equation}\label{def:kappa}
\kappa:=\frac{c_{Q}}{c_{1}}\int Q^{p}(x)e^{-x}\ud x>0\quad\mbox{and}\quad c_{1}:=\|Q'\|_{L^{2}}^{2}.
\end{equation}
\item\emph{Leading order interactions.} Let any $1<\theta<\min (p-1,\frac{3}{2})$.
\begin{itemize}
\item If $K=1$ then $G=0\,;$
\item If $K\geq 2$ then
\begin{gather}
\left|\langle G,\partial_{x}Q_{1}\rangle+c_{1}\kappa\sigma_{1}\sigma_{2} e^{-(z_{2}-z_{1})} \right|
\lesssim \sum_{l=1}^{K-1}e^{-\theta(z_{l+1}-z_{l})};\label{on:G1}\\
\left|\langle G,\partial_{x}Q_{K}\rangle-c_{1}\kappa\sigma_{K-1}\sigma_{K} e^{-(z_{K}-z_{K-1})}\right|
\lesssim \sum_{l=1}^{K-1}e^{-\theta(z_{l+1}-z_{l})};\label{on:GK}
\end{gather}
\item If $K\geq 3$ then, for any $k\in \{2,\ldots,K-1\}$,
\begin{equation}\label{on:G}\begin{aligned}
&\left|\langle G,\partial_{x}Q_{k}\rangle-c_{1}\kappa \sigma_{k} \left[\sigma_{k-1}e^{-(z_{k}-z_{k-1})}-\sigma_{k+1} e^{-(z_{k+1}-z_{k})}\right]\right|\\ 
&\quad\lesssim \sum_{l=1}^{K-1}e^{-\theta(z_{l+1}-z_{l})}.
\end{aligned}\end{equation}
\end{itemize}
\end{enumerate}
\end{lemma}
\begin{proof}
Proof of (i). These estimates are direct consequences of the 
decay properties of $Q$ in \eqref{asym:Q} and $p>2$. See details in \cite[proof of Lemma 2.1]{CMYZ}.

Proof of (ii). We claim the following estimate for $z\gg1$.
\begin{equation}\label{est:Qp}
\bigg|\int Q^{p}(y)Q(y+z)\ud y-c_{1}\kappa e^{-z}\bigg|\lesssim e^{-\frac{3}{2}z}.
\end{equation}
Observe that~\eqref{tech4} follows directly from~\eqref{est:Qp}. 

Now, we prove~\eqref{est:Qp}. First, for $|y|<\frac{3}{4}z$, using~\eqref{asym:Q}, we have
\begin{equation*}
\big|Q(y+z)-c_{Q}e^{-(y+z)}\big|\lesssim e^{-2(y+z)}\lesssim e^{-2z}e^{2|y|},
\end{equation*}
and so
\begin{equation*}
\bigg|\int _{|y|<\frac{3}{4}z}Q^{p}(y)\big[Q(y+z)-c_{Q}e^{-(y+z)}\big]\ud y \bigg|\lesssim e^{-2z}.
\end{equation*}
Second, using~\eqref{asym:Q} and $p>2$, it holds
\begin{equation*}
\begin{aligned}
\int _{|y|>\frac{3}{4}z}Q^{p}(y)Q(y+z)\ud y+\int _{|y|>\frac{3}{4}z}Q^{p}(y)e^{-(y+z)}\ud y
\lesssim e^{-\frac{3}{2}z}.
\end{aligned}
\end{equation*}
Gathering these estimates, we have proved~\eqref{est:Qp}.

Proof of (iii). We treat the case $k\in \{2,\ldots,K-1\}$ for $K\geq 3$. 
Other cases are similar.
On the one hand, using Taylor formula, it holds
\begin{equation*}
\bigg|G-p|Q_{k}|^{p-1} \sum_{k'\ne k}Q_{k'} \bigg|\lesssim |Q_{k}|^{p-2} \sum_{k'\ne k}|Q_{k'}|^{2}+
\sum_{k'\neq k, l\neq k'} |Q_{k'}|^{p-1}|Q_l| .
\end{equation*}
Thus, using~\eqref{tech1}, we have for any $1<\theta<\min(p-1,2)$,
\begin{align*}
\bigg|\langle G,\partial_{x}Q_{k}\rangle-\sum_{k'\ne k}\langle p|Q_{k}|^{p-1}Q_{k'},\partial_{x}Q_{k}\rangle\bigg| 
&\lesssim\sum_{l\neq k'}\int |Q_{k'}|^{p-1}|Q_{l}|^{2} \ud x
\\&\lesssim \sum_{l=1}^{K-1}e^{-\theta(z_{l+1}-z_{l})}.
\end{align*}
On the other hand, by direct computation, integrating by parts and using the proof of~\eqref{est:Qp}, we obtain, for any $1<\theta <\min (p-1,\frac{3}{2})$,
\begin{align*}
&\sum_{k'\ne k}\langle p|Q_{k}|^{p-1}Q_{k'},\partial_{x}Q_{k} \rangle\\
&=-\sum_{k'\ne k}\sigma_{k}\sigma_{k'}\int Q^{p}(y)\partial_{x}Q(y+z_{k}-z_{k'})\ud y \\
&=\sum_{k'<k}\sigma_{k'}\sigma_{k}c_{1}\kappa e^{-(z_{k}-z_{k'})}
-\sum_{k'>k}\sigma_{k}\sigma_{k'}c_{1}\kappa e^{-(z_{k'}-z_{k})}+O\bigg(\sum_{l=1}^{K-1}e^{-\frac 32(z_{l+1}-z_{l})}\bigg)\\
&=\sigma_{k-1}\sigma_{k}c_{1}\kappa e^{-(z_{k}-z_{k-1})}-\sigma_{k}\sigma_{k+1}c_{1}\kappa e^{-(z_{k+1}-z_{k})}+O\bigg(\sum_{l=1}^{K-1}e^{-\frac 32(z_{l+1}-z_{l})}\bigg).
\end{align*}
Indeed, for example, if $k\geq 3$, one has 
\begin{equation}\label{eq:exa}z_k-z_{k-2} =( z_k-z_{k-1}) + (z_{k-1}-z_{k-2})
\geq 2 \min(z_k-z_{k-1};z_{k-1}-z_{k-2}),\end{equation}
which allows us to consider interactions between to non neighbor solitary waves as error terms.
Gathering the above estimates, we find~\eqref{on:G}.
\end{proof}
\subsection{Decomposition close to the sum of solitary waves}

\begin{lemma}\label{le:dec}
Let $\vec u=(u,\partial_t u)$ be a solution of~\eqref{nlkg} on the interval $[T_1 ,T_2 ]$ and let $K\ge 1$.
Assume that 
\begin{equation} \label{for:dec}
\sup_{t\in [T_1 ,T_2 ]}\bigg\{ \inf_{ \xi_{k+1}-\xi_k>|\log \gamma|} \bigg\|u(t)-\sum_{k=1}^{K}\sigma_{k} Q(\cdot -{\xi}_{k})\bigg\|_{H^{1}}+\|\partial_{t} u(t)\|_{L^{2}}\bigg\}<\gamma,
\end{equation}
for some small $\gamma>0$.
Then, there exist unique $\mathcal C^1$ functions
\begin{equation*}
t\in [T_1 ,T_2 ]\mapsto \left( z_k (t) ,\ell_k (t)\right)_{k\in\{1,\ldots,K\}}\in
\mathbb R^{2K},
\end{equation*}
such that the solution $\vec{u}$ decomposes as
\begin{equation}\label{def:ee}
\vec u = \begin{pmatrix} u \\ \partial_t u\end{pmatrix}=
\sum_{k=1}^{K}\vec Q_k + \vve,\quad
\vve=\begin{pmatrix}\e \\ \eta \end{pmatrix}
\end{equation}
with the following properties on $[T_1 , T_2 ]$.

\begin{enumerate}
\item \emph{Orthogonality and smallness.} For any $k=1,\ldots, K$,
\begin{equation}\label{ortho}
\langle \e, \partial_{x}Q_{k}\rangle=\langle \eta,\partial_{x}Q_{k}\rangle=0
\end{equation}
and
\begin{equation*}
\|\vve\|_\ENE +\sum_{l=1}^{K} |\ell_l|+\sum_{l=1}^{K-1}e^{-2(z_{l+1}-z_{l})}\lesssim\gamma.
\end{equation*}

\item \emph{Equation of $\vve$.}
\begin{equation}\label{syst_e}\left\{\begin{aligned}
\partial_t \e & = \eta + \md _{{\mathbf \e}}\\
\partial_t \eta &
= \partial_x^2 \e-\e+f(R+\e)-f(R)
-2\alpha\eta + \md_{\eta}+G
\end{aligned}\right.\end{equation}
where
\begin{align*}
&\mathrm{Mod}_{\e}=
\sum_{k=1}^{K}\left(\dot z_{k}-
{\ell}_{k}\right)\partial_{x}Q_{k},\\
&\mathrm{Mod}_{\eta}=
\sum_{k=1}^{K}\big(\dot{\ell}_{k}+2\alpha{\ell}_{k}\big)\partial_{x}Q_{k}-\sum_{k=1}^{K} {\ell}_{k}\dot{z}_{k}\partial_{x}^{2}Q_{k}.
\end{align*}
\item \emph{Control of the geometric parameters.} For $k=1,\ldots,K$,
\begin{align}
|\dot{z}_{k}-{\ell}_{k}|&\lesssim \|\vve \|^{2}_\ENE+\sum_{l=1}^{K}|\ell_{l}|^{2},
\label{eq:z}\\
|\dot{{\ell}}_{k}+2\alpha{\ell}_{k}|
&\lesssim \|\vve\|_\ENE^{2}+\sum_{l=1}^{K}|\ell_{l}|^{2}+\sum_{l=1}^{K-1}e^{-(z_{l+1}-z_{l})} .\label{eq:l}
\end{align}
\item\emph{Control of the exponential directions.} 
For $k=1,\cdots, K$, let 
\begin{equation}\label{def:akpm}
a_k^{\pm} = \langle \vve,\vec Z_{k}^{\pm}\rangle.
\end{equation}
Then,
\begin{equation}\label{eq:a}
\left| \frac \ud{\ud t} a_k^{\pm}- \nu^\pm a_{k}^{\pm}\right|\lesssim \|\vve \|_\ENE^{2}+\sum_{l=1}^{K}|{\ell}_{l}|^{2}+ \sum_{l=1}^{K-1}e^{-(z_{l+1}-z_{l})}.
\end{equation}
\end{enumerate}
\end{lemma}
\begin{remark}\label{rk:KK01}
In the above estimates, if $K=1$ then the terms $\sum_{l=1}^{K-1}$ are $0$.
\end{remark}
\begin{proof}
See \cite[proof of Lemma 2.2]{CMYZ}.
\end{proof}
\begin{lemma} [Refined equation for $\ell_{k}$] 
In the context of Lemma~\ref{le:dec}, assume that $K\geq 2$. 
Let $1<\theta<\min(p-1,\frac{3}{2})$. Then,
\begin{multline}\label{eq:lbis1}
 \Big| \dot{{\ell}}_{1}+2\alpha{\ell}_{1}-\kappa\sigma_{1}\sigma_{2}e^{-(z_{2}-z_{1})}\Big|
 +
 \Big| \dot{{\ell}}_{K}+2\alpha{\ell}_{K}+\kappa \sigma_{K-1}\sigma_{K}e^{-(z_{K}-z_{K-1})}\Big|\\
 \lesssim \|\vve\|_\ENE^{2}+\sum_{l=1}^{K}|{\ell}_{l}|^{2}+\sum_{l=1}^{K-1}e^{-\theta (z_{l+1}-z_{l})}.
\end{multline}
Moreover, if $K\geq 3$ then for any $k\in \{2,\ldots,K-1\}$,
\begin{multline}\label{eq:lbis}
 \Big| \dot{{\ell}}_{k}+2\alpha{\ell}_{k}+\kappa \sigma_{k} \left[\sigma_{k-1}e^{-(z_{k}-z_{k-1})}-\sigma_{k+1}e^{-(z_{k+1}-z_{k})}\right]\Big|\\
 \lesssim \|\vve\|_\ENE^{2}+\sum_{l=1}^{K}|{\ell}_{l}|^{2}+\sum_{l=1}^{K-1}e^{-\theta (z_{l+1}-z_{l})}.
\end{multline}
\end{lemma}
\begin{proof}
We treat the case $k\in \{2,\ldots,K-1\}$ for $K\geq 3$. 
Other cases are similar.
First,
\begin{equation*}
\frac{\ud}{\ud t}\langle \eta,\partial_{x}Q_{k} \rangle= \langle \partial_t\eta, \partial_{x}Q_{k}\rangle+\langle \eta,\partial_{t}\partial_{x}Q_{k}\rangle=0.
\end{equation*}
Thus, using~\eqref{ortho} and~\eqref{syst_e},
\begin{align*}
0=
&\langle \partial_{x}^{2}\e-\e+f'(Q_{k})\e,\partial_{x}Q_{k}\rangle +\langle f(R+\e)-f(R)-f'(R)\e, \partial_{x}Q_{k}\rangle\\
&+\langle \big(f'(R)-f'(Q_{k})\big)\e, \partial_{x}Q_{k}\rangle +\langle G,\partial_{x}Q_{k}\rangle +\langle {\rm{Mod}}_{\eta}, \partial_{x}Q_{k}\rangle-\langle \eta, \dot{z}_{k}\partial_{x}^{2}Q_{k}\rangle.
\end{align*}
Since $\partial_{x}Q_k$ satisfies $\partial^{2}_{x}\partial_{x}Q_k-\partial_{x}Q_k+f'(Q_k )\partial_{x}Q_k=0$, by integration by parts, the first term is zero. Next, by Taylor expansion (using $p>2$), we have
the pointwise estimate
\begin{equation*}
\big|f(R+\e)-f(R)-f'(R)\e\big|\lesssim |\e|^{p}+ |\e|^{2} \sum_{l=1}^{K} |Q_{l}|^{p-2} 
\end{equation*}
and
\begin{equation*}
\big|f'(R)-f'(Q_{k})\big|\lesssim |Q_{k} |^{p-2} \sum_{k'\ne k} |Q_{k'} | +\sum_{k'\ne k}\big|Q_{k'}\big|^{p-1}.
\end{equation*}
Thus, using $\|\cdot\|_{L^\infty}\lesssim \|\cdot \|_{H^1}$, 
\begin{align*}
\left|\langle f(R+\e)-f(R)-f'(R)\e,\partial_{x} Q_k \rangle\right|\lesssim \|\e\|_{H^1}^{2}
\end{align*}
and by the Cauchy-Schwarz inequality and \eqref{tech2},
\begin{equation*}
\left|\langle (f'(R)-f'(Q_{k}) )\e, \partial_{x}Q_{k}\rangle\right|\lesssim \|\e\|_{H^1}^{2}
+\sum_{k=1}^{K-1}e^{-2(z_{k+1}-z_{k})}\ .
\end{equation*}
By direct computation, we obtain
\begin{align*}
\langle {\rm{Mod}}_{\eta}, \partial_{x}Q_{k}\rangle
&=\big(\dot{\ell}_{k}+2\alpha\ell_{k}\big)\|Q'\|_{L^{2}}^{2}
+\sum_{k'\ne k}\big(\dot{\ell}_{k'}+2\alpha\ell_{k'}\big)\langle \partial_{x}Q_{k'},\partial_{x}Q_{k}\rangle\\
&\quad -\sum_{k'\ne k}\big(\ell_{k'}\dot{z}_{k'}\big)\langle \partial^{2}_{x}Q_{k'},\partial_{x}Q_{k}\rangle.
\end{align*}
Thus, using the equation of $Q$,~\eqref{tech1}, \eqref{def:kappa} and~\eqref{eq:z}, we obtain
\begin{align*}
\langle {\rm{Mod}}_{\eta}, \partial_{x}Q_{k}\rangle
&=c_1\big(\dot{\ell}_{k}+2\alpha\ell_{k}\big)
+O\bigg(\sum_{k'\ne k}|\dot{\ell}_{k'}+2\alpha\ell_{k'}|e^{-\frac{3}{4}|z_{k'}-z_{k}|}\bigg)\\
&\quad +O\bigg(\|\vec{\varepsilon}\,\|^{2}_{\ENE}+\sum_{l=1}^{K}|\ell_{l}|^{2}\bigg).
\end{align*}
Note that, by~\eqref{eq:z} and the Cauchy-Schwarz inequality, we obtain
\begin{equation*}
\big|\langle \eta, \dot{z}_{k}\partial_{x}^{2}Q_{k}\rangle\big|\lesssim \big(|\dot{z}_{k}-\ell_{k}|+|\ell_{k}|\big)\|\vec{\varepsilon}\,\|_{\ENE}\lesssim \|\vec{\varepsilon}\,\|_{\ENE}^{2}+\sum_{k=1}^{K}|\ell_{k}|^{2}.
\end{equation*}
Gathering above estimates and using~\eqref{on:G1}-\eqref{on:GK}-\eqref{on:G}, we obtain
\begin{equation}\label{eq:finell}
\begin{aligned}
&\Big| \dot{{\ell}}_{k}+2\alpha{\ell}_{k}+\kappa\sigma_{k}\left[ \sigma_{k-1}e^{-(z_{k}-z_{k-1})}- \sigma_{k+1}e^{-(z_{k+1}-z_{k})}\right]\Big|\\
&\lesssim \|\vve\|_\ENE^{2}+\sum_{l=1}^{K}|{\ell}_{l}|^{2}+\sum_{k'\ne k}|\dot{\ell}_{k'}+2\alpha\ell_{k'}|e^{-\frac{3}{4}|z_{k'}-z_{k}|}+\sum_{l=1}^{K-1}e^{-\theta (z_{l+1}-z_{l})}.
\end{aligned}
\end{equation}
We obtain~\eqref{eq:lbis} by combining~\eqref{eq:finell} for all $k\in \{1,\ldots,K\}$.
\end{proof}
\subsection{Energy estimates}\label{S.3.3}
For $\mu>0$ small to be chosen, we denote $\rho=2\alpha-\mu$. Consider the nonlinear energy functional
\begin{equation*} 
\mathcal{E} =
\int \big\{ (\partial_x\e)^2+ (1-\rho\mu )\e^{2}
+(\eta+\mu\e)^2- 2 [F (R +\e )-F (R )-f (R )\e]\big\}.
\end{equation*}
We recall the following energy estimates.
\begin{lemma}\label{le:ener} There exists $\mu>0$ such that
in the context of Lemma~\ref{le:dec},
 the following hold.
\begin{enumerate}
\item\emph{Coercivity and bound.}
\begin{equation}\label{eq:coer}
\mu \|\vve \|_\ENE^{2}-\frac{1}{2\mu}\sum_{k=1}^{K}\left( (a_{k}^{+})^{2}+(a_{k}^{-})^{2}\right)
\leq \mathcal{E}\leq \frac 1\mu \|\vve \|_\ENE^{2}.
\end{equation}
\item\emph{Time variation.}
\begin{equation}\label{eq:E}
\frac \ud{\ud t}\mathcal{E}\le -2\mu\mathcal{E}
+\frac{1}{\mu}\|\vve \|_\ENE \bigg[\|\vve \|_\ENE^{2}+\sum_{k=1}^{K}|\ell_{k}|^{2}+\sum_{k=1}^{K-1}e^{-(z_{k+1}-z_{k})}\bigg].
\end{equation}
\end{enumerate}
\end{lemma}
\begin{remark}
The above lemma is valid for any $\mu>0$ small enough. For future needs, we fix such a $\mu>0$ satisyfing the following additional smallness condition
\begin{equation}\label{on:mu}
\mu \leq \min\left(1,\alpha, |\nu_-|\right).
\end{equation}
\end{remark}
\begin{remark}\label{rk:K01}
In the case where $K=1$, the term $\sum_{k=1}^{K-1}$ due to the nonlinear interactions does not appear in~\eqref{eq:E}.
\end{remark}
\begin{proof}
See \cite[proof of Lemma 2.4]{CMYZ}.
\end{proof}

\subsection{Time evolution analysis}\label{S:2.4}
We introduce new parameters and functionals to analyse the time evolution of solutions in the framework of Lemma~\ref{le:dec}.
First, we set, for $k=1,\ldots,K$,
\begin{equation*}
y_{k}=z_{k}+\frac{\ell_{k}}{2\alpha},
\end{equation*}
and for $k=1,\ldots,K-1$ (when $K\geq 2$),
\begin{equation*}
r_{k}=y_{k+1}-y_{k} \gg 1.
\end{equation*}
Second, we define
\begin{gather*}
\mathcal{K}_{+}=\{k=1,\ldots,K-1:\sigma_{k}=\sigma_{k+1}\},\quad F_{+}=\sum_{k\in \mathcal{K}_{+}}e^{-r_{k}},
\\
\mathcal{K}_{-}=\{k=1,\ldots,K-1:\sigma_{k}=-\sigma_{k+1}\},\quad F_{-}=\sum_{k\in \mathcal{K}_{-}}e^{-r_{k}}.
\end{gather*}
Third, we introduce notation for the damped components of the solution
\begin{equation*}
\mathcal F = \mathcal E + \BB,\quad \BB= \sum_{k=1}^{K} |\ell_k|^2 + \frac 1{2 \mu} \sum_{k=1}^{K} (a_k^-)^2 ,
\end{equation*}
and for all the components except distances
\begin{equation*}
\mathcal N = \bigg[ \|\vve\|_\ENE^2 + \sum_{k=1}^{K} |\ell_k|^2 \bigg]^{\frac 12}.
\end{equation*}
Last, we define
\begin{equation*}
b=\sum_{k=1}^{K} (a_k^+)^2,\quad \mathcal M= \frac 1{\mu^2}\left(\mathcal F - \frac{b}{2\nu^+}\right).
\end{equation*}
We rewrite the estimates of Lemmas~\ref{le:dec} and~\ref{le:ener} using such notation.

\begin{lemma}\label{le:new} Assume $K\geq 2$.
Let any $1<\theta<\min (p-1,\frac{5}{4})$.
In the context of Lemma~\ref{le:dec}, the following hold.
\begin{enumerate}
\item \emph{Comparison with original variables.} For $k=1,\ldots,K-1$,
\begin{gather}
 \big| r_{k}-(z_{k+1}-z_{k}) \big| \lesssim \mathcal N, \label{eq:new1} \\
 \mu \mathcal N^2\leq 
 \mu \|\vve\|_\ENE^2 + \sum_{l=1}^{K} |\ell_l|^2 \leq \mathcal F + \frac b{2\mu}
\lesssim\mathcal N^2.\label{eq:new2}
\end{gather}
\item\emph{ODE for the distances between solitary waves.} 
The equation for the evolution of $y_k$ is
\begin{equation} \label{eq:y_k}
\left\{ \begin{aligned}
&\dot{y}_{1}=\frac{\kappa}{2\alpha}\sigma_{1}\sigma_{2}e^{-r_{1}}+O\big(\mathcal{N}^{2}+F^{\theta}_{+}+F_{-}^{\theta}\big),\\
&\dot{y}_{k}=-\frac{\kappa}{2\alpha}\sigma_{k-1}\sigma_{k}e^{-r_{k-1}}+\frac{\kappa}{2\alpha}\sigma_{k}\sigma_{k+1}e^{-r_{k}}+O\big(\mathcal{N}^{2}+F^{\theta}_{+}+F_{-}^{\theta}\big),\\
&\dot{y}_{K}=-\frac{\kappa}{2\alpha}\sigma_{K-1}\sigma_{K}e^{-r_{K-1}}+O\big(\mathcal{N}^{2}+F^{\theta}_{+}+F_{-}^{\theta}\big),\end{aligned}
\right.
\end{equation}
for $2\le k\le K-1$.
Moreover, there exists $\lambda>0$ such that
\begin{align}
\frac{\ud}{\ud t}\left[\frac{1}{F_{+}}\right]&\le -\lambda+\frac{1}{\lambda F_{+}}\left(\mathcal{N}^{2}+F^{\theta}_{+}+F^{\theta}_{-}\right),\label{eq:dist+}\\
\frac{\ud}{\ud t}\left[\frac{1}{F_{-}}\right]&\ge \lambda-\frac{1}{\lambda F_{-}}\left(\mathcal{N}^{2}+F^{\theta}_{+}+F^{\theta}_{-}\right).\label{eq:dist-}
\end{align}
\item\emph{Exponential instability.}
\begin{equation}\label{eq:b}
 | \dot b - 2 \nu^+ b |\lesssim \mathcal N^3+ \mathcal N (F_{+}+F_{-}).
\end{equation}
\item\emph{Damped components.}
\begin{equation}\label{eq:damped}\begin{aligned}
&\frac \ud{\ud t}\mathcal F + 2 \mu \mathcal F \lesssim \mathcal N^3+ \mathcal N (F_{+}+F_{-}),
\\&\frac \ud{\ud t}\BB + 2 \mu \BB \lesssim \mathcal N^3+ \mathcal N (F_{+}+F_{-}).
\end{aligned}\end{equation}
\item\emph{Liapunov type functional.} 
\begin{equation} \label{eq:Mbis}
\frac \ud{\ud t} \mathcal M \leq -\mathcal N^2 + C \big(F^{2}_{+}+F^{2}_{-}\big).
\end{equation}
\item\emph{Refined estimates for the distance.}
Setting
\begin{equation*}
R_+=\frac1{F_{+}} \exp\left( \lambda^{-1}\mathcal M \right) \quad \text{and}\quad
R_-=\frac1{F_{-}} \exp\left(- 3\lambda^{-1}\mathcal M \right),
\end{equation*}
where $\lambda$ is given in $({\rm ii})$, it holds
\begin{align}
\frac{\ud}{\ud t} R_+ &\leq \left(-\lambda+\frac{2}{\lambda F_{+}}\left(F_{+}^{\theta}+F_{-}^{\theta}\right)\right)\exp\left( \lambda^{-1}\mathcal M \right),\label{eq:RF+}\\
\frac{\ud}{\ud t} R_- &\geq \left(\lambda+\frac{2}{\lambda F_{-}}\mathcal{N}^{2}-\frac{2}{\lambda F_{-}}\left(F_{+}^{\theta}+F_{-}^{\theta}\right)\right)\exp\left(- 3\lambda^{-1}\mathcal M \right)\label{eq:RF-}.
\end{align} 
\end{enumerate}
\end{lemma}
\begin{remark}\label{rk:KK}
When using the quantities $F_\pm$ and $R_\pm$, we tacitly assume that $\mathcal K_\pm$ is not empty.
Otherwise, the respective quantities are ignored. For example, if the set $\mathcal K_+$ is empty, then
\eqref{eq:RF-} rewrites:
\begin{equation*}
\frac{\ud}{\ud t} R_- \geq \left(\lambda+\frac{2}{\lambda}\frac{\mathcal{N}^{2}}{F_{-}}-\frac{2}{\lambda } F_{-}^{\theta-1} \right)\exp\left(- 3\lambda^{-1}\mathcal M \right).
\end{equation*}
When $K=1$, the quantities $F_\pm$ and $R_\pm$ are systematically $0$ and the variables
$r_k$ are not defined. The estimates~\eqref{eq:new2}, \eqref{eq:b}, \eqref{eq:damped} and \eqref{eq:Mbis} hold in this context.
\end{remark}

\begin{proof}
Proof of (i): \eqref{eq:new1} follows from triangle inequality and the definition of $\mathcal N$ and \eqref{eq:new2} follows directly from~\eqref{eq:coer}.

Proof of (ii).
In this computation, we use the convention that terms involving $y_k$ or $\sigma_k$ for $k\leq 0$ or
$k\geq K+1$ are zero, for example by setting $\sigma_0=\sigma_{K+1}=0$.
Similarly, when $e^{-r_{0}}$ and $e^{-r_{K}}$ appears, the corresponding term has to be ignored.
By direct computation and using~\eqref{eq:z},~\eqref{eq:lbis1}-\eqref{eq:lbis}, we obtain for $k=1,\ldots,K$,
\begin{equation*}
\dot{y}_{k}=\dot{z}_{k}+\frac{\dot\ell_{k}}{2\alpha}=-\frac{\kappa}{2\alpha}\sigma_{k-1}\sigma_{k}e^{-r_{k-1}}+\frac{\kappa}{2\alpha}\sigma_{k}\sigma_{k+1}e^{-r_{k}}+O\big(\mathcal{N}^{2}+F^{\theta}_{+}+F_{-}^{\theta}\big).
\end{equation*}
It follows that, for any $k=1,\ldots,K-1$,
\begin{align*}
\dot{r}_{k}=\dot{y}_{k+1}-\dot{y}_{k}
&=-\frac{\kappa}{\alpha}\sigma_{k}\sigma_{k+1}e^{-r_{k}}+\frac{\kappa}{2\alpha}\sigma_{k+1}\sigma_{k+2}e^{-r_{k+1}}+\frac{\kappa}{2\alpha}\sigma_{k-1}\sigma_{k}e^{-r_{k-1}} \\
&\quad +O\big(\mathcal{N}^{2}+F_{+}^{\theta}+F_{-}^{\theta}\big). 
\end{align*}
On the right-hand side of the above expression, the first term is always present for $k=1,\ldots,K-1$,
while the second and third terms might be zero depending on the value of $k$.
For $k\in\mathcal{K}_{+}$, it holds $\sigma_k=\sigma_{k+1}$ and one sees that
\begin{equation*}
\dot{r}_{k}=-\frac{\kappa}{\alpha}e^{-r_{k}}+\frac{\kappa}{2\alpha}\sigma_{k+1}\sigma_{k+2}e^{-r_{k+1}}+\frac{\kappa}{2\alpha}\sigma_{k-1}\sigma_{k}e^{-r_{k-1}} +O\big(\mathcal{N}^{2}+F_{+}^{\theta}+F_{-}^{\theta}\big),
\end{equation*}
with the same observation concerning the second and third terms on the right-hand side.
Thus,
\begin{equation}\label{eq:step}
\sum_{k\in \mathcal{K}_{+}}\dot{r}_{k}e^{-r_{k}}
 =-\frac{\kappa}{2 \alpha} S_+ 
 +O\big(F_{+} \big(\mathcal{N}^{2}+F^{\theta}_{+}+F_{-}^{\theta}\big)\big).
\end{equation}
where $S_+$ denotes
\[
S_+=\sum_{k\in \mathcal{K}_{+}}\left[2 e^{-2r_{k}}
-\sigma_{k+1}\sigma_{k+2}e^{-(r_{k}+r_{k+1})}- \sigma_{k-1}\sigma_{k}e^{-(r_{k-1}+r_{k})}
\right].
\]
We claim that there exists $\tilde \lambda>0$ such that $S_+$ satisfies
\begin{equation}\label{on:Sp}
S_+\geq \tilde \lambda \sum_{k\in \mathcal{K}_{+}} e^{-2r_{k}} .
\end{equation}
Indeed, first, recall that the symmetric matrix of size $n$
\[
A_n= \begin{pmatrix}
2 & -1 & 0 & \ldots & 0 \\
-1 & 2 & -1 & \ddots & \vdots \\
0 &-1 & 2 & \ddots & 0 \\
\vdots & \ddots & \ddots & \ddots & -1 \\
0 & \ldots& 0 & -1 & 2 
\end{pmatrix}
\]
is definite positive by the Sylvester criterion since 
for any $j\in \{1,\ldots n\}$, the $j$th leading principal minor of this matrix, 
\emph{i.e.} the determinant of its upper-left $j\times j$ sub-matrix, is positive
(its value is $j+1$).

Second, observe that in the sum defining $S_+$, for given $k\in \mathcal K_+$, if
$k-1\not \in \mathcal K_+$, then $\sigma_{k-1}\sigma_{k}=1$ (if $k\geq 2$) or $0$ (if $k=1$)
and thus the corresponding term is positive or zero and can be ignored in establishing a lower bound for $S_+$. The same property is true for the term 
corresponding to $\sigma_{k+1}\sigma_{k+2}$ if $k+1\not\in \mathcal K_+$. 
Letting $n=\card(\mathcal K_+)$, this observation justifies that $S_+$ is lower bounded by the quadratic form associated 
to the matrix $A_n$ taken at the vector of $\RR^n$ of components $\{e^{-r_k}: k\in \mathcal K_+\}$.
This shows that \eqref{on:Sp} holds for some $\tilde \lambda>0$.

It follows from \eqref{eq:step} and \eqref{on:Sp} that there exists $\lambda>0$
such that
\begin{equation*}
\sum_{k\in \mathcal{K}_{+}}\dot{r}_{k}e^{-r_{k}} \le -\lambda F^{2}_{+}+\frac 1\lambda F_{+}\big(\mathcal{N}^{2}+F^{\theta}_{+}+F_{-}^{\theta}\big).
\end{equation*}
Thus,
\begin{equation*}
\frac{\ud}{\ud t}\left[\frac{1}{F_{+}}\right]=\frac1{F_{+}^2} \sum_{k\in \mathcal{K}_{+}}\dot{r}_{k}e^{-r_{k}}\le -\lambda+\frac{1}{\lambda F_{+}}\big(\mathcal{N}^{2}+F^{\theta}_{+}+F_{-}^{\theta}\big).
\end{equation*}
The estimate \eqref{eq:dist-} concerning $\frac{\ud}{\ud t} [\frac{1}{F_{-}} ]$ is proved similarly.

Proof of (iii). It follows from~\eqref{eq:a} and the bound $|a_{k}^{+}|\lesssim \|\vec{\varepsilon}\,\|_{\ENE}\leq \mathcal{N}.$

Proof of (iv). By the definition of $\mathcal{F}$ and~\eqref{eq:l}, \eqref{eq:a}, \eqref{eq:E}, we have
\begin{align*}
\frac{\ud}{\ud t}\mathcal{F}
&=\frac{\ud }{\ud t}\mathcal{E}+2\sum_{k=1}^{K}\dot{\ell}_{k}\ell_{k}+\frac{1}{\mu}\sum_{k=1}^{K}\dot{a}^{-}_{k}a^{-}_{k}\\
&\le -2\mu \mathcal{E}-4\alpha\sum_{k=1}^{K}\ell_{k}^{2}+\frac{\nu^{-}}{\mu}\sum_{k=1}^{K}(a_{k}^{-})^{2}+O\big(\mathcal{N}^{3}+\mathcal{N}\big(F_{+}+F_{-}\big)\big).
\end{align*}
Since $0<\mu\leq\alpha$ and $0<\mu \leq|\nu^{-}|$ (see \eqref{on:mu}), we obtain~\eqref{eq:damped} for $\mathcal{F}$. The proof for $\mathcal{G}$ is the same.

Proof of (v). From~\eqref{eq:b} and~\eqref{eq:damped}, we estimate
\begin{equation*}
\frac{\ud}{\ud t}\mathcal{M}=\frac{1}{\mu^{2}}\left(\frac{\ud }{\ud t}\mathcal{F}-\frac{1}{2\nu^{+}} \frac{\ud }{\ud t}b\right)
\le -\frac{2}{\mu} \left(\mathcal{F}+\frac{b}{2\mu}\right)+O\big(\mathcal{N}^{3}+\mathcal{N}\big(F_{+}+F_{-}\big) \big).
\end{equation*}
Thus, from~\eqref{eq:new2} and then for $\mathcal{N}$ small enough,
\begin{equation*}
\frac{\ud}{\ud t}\mathcal{M}
\le -2\mathcal{N}^{2}+O\big(\mathcal{N}^{3}+\mathcal{N}\big(F_{+}+F_{-}\big) \big)
\leq -\mathcal{N}^{2}+O\big( F_{+} ^2+F_{-}^2\big).
\end{equation*}

Proof of (vi). From~\eqref{eq:dist+} and~\eqref{eq:Mbis}, we estimate
\begin{align*}
\frac{\ud }{\ud t}R_+
&=\left(\frac{\ud}{\ud t}\left[\frac{1}{F_{+}}\right]+\frac{1}{\lambda F_{+}}\frac{\ud }{\ud t}\mathcal{M}\right)\exp (\lambda^{-1}\mathcal{M}) \\
&\le \left(-\lambda +\frac{1}{\lambda F_{+}}\big(F^{\theta}_{+}+F_{-}^{\theta}+CF^{2}_{+}+CF^{2}_{-}\big)
\right)\exp (\lambda^{-1}\mathcal{M}) ,
\end{align*}
which implies~\eqref{eq:RF+}. The estimate for $\frac{\ud}{\ud t}R_-$ is proved similarly.
Note that the coefficient~$3$ of the factor $\exp (-3\lambda^{-1}\mathcal{M})$ in the definition $R_-$
allows us to obtain a positive factor $\mathcal N^2$ in the right-hand side of \eqref{eq:RF-}.
\end{proof}

\subsection{Long-time energy asymptotics}

\begin{lemma}
Let any $1<\theta<\min (p-1,\frac{5}{4})$.
 In the context of Lemmas~\ref{le:dec} and~\ref{le:new},
it holds
\begin{equation}\label{eq:E:R}
E(\vec u)=KE(Q,0)-c_{1}\kappa F_{+} +c_{1}\kappa F_{-}+O(\mathcal{N}^2+F^{\theta}_{+}+F^{\theta}_{-}).
\end{equation}
\end{lemma}
\begin{proof}
Expanding $E(u,\partial_t u)$ using the decomposition~\eqref{def:ee}, integration by parts,
the equation $-\partial_{x}^{2} Q+Q-f(Q)=0$ and the definition of $G$ in~\eqref{def:G}, we find
\begin{align*}
2 E(u,\partial_{t}u)
&= \int \left|\partial_{t}u\right|^{2} + 2 E\left(R,0\right)
-2\int G\e \\
&\quad
+ \int \left(|\partial_{x}\e|^2 + \e^2 -2 F(R+\e)+2 F(R)+2f(R)\e\right).
\end{align*}
Thus, using \eqref{tech2}, the Cauchy-Schwarz and Sobolev inequalities, it holds
\begin{equation*}
2 E(u,\partial_{t}u)
=\int \left|\partial_{t}u\right|^{2} + 2 E\left(R,0\right)
+O \big(\mathcal{N}^2+F^{2}_{+}+F^{2}_{-}\big).
\end{equation*}
Note also that~\eqref{def:ee} implies
\begin{equation*}
\int \left|\partial_{t}u\right|^{2}\lesssim \int \bigg( \left|\eta\right|^{2}+\sum_{k=1}^{K}\big|\ell_{k}\partial_{x}Q_{k}\big|^{2}\bigg)\lesssim \mathcal{N}^{2}.
\end{equation*}
Then, by direct computation, next $-\partial_{x}^{2} Q+Q-f(Q)=0$ and \eqref{tech3},
\begin{align*}
E(R,0)
&=KE(Q,0)+\sum_{k<k'}\int \left[(\partial_{x} Q_k) (\partial_{x} Q_{k'})+Q_k Q_{k'} -f(Q_k )Q_{k'} -f(Q_{k'} )Q_k \right] \\
&\quad-\int \bigg(F(R)-\sum_{k=1}^{K}F(Q_k )-\sum_{k\ne k'}f(Q_k )Q_{k'} \bigg) \\
&=KE(Q,0)-\sum_{k< k'}\langle f(Q_k), Q_{k'}\rangle+ O (F^{\theta}_{+}+F^{\theta}_{-}).
\end{align*}
Last, from~\eqref{tech4},~\eqref{eq:new1} and the definition of $F_{+}$ and $F_{-}$, we observe that
\begin{equation*}
\sum_{k<k'}\langle f(Q_{k}), Q_{k'}\rangle=c_{1}\kappa F_{+}-c_{1}\kappa F_{-}+O(\mathcal{N}^{2}+F^{\theta}_{+}+F^{\theta}_{-})
\end{equation*}
Indeed, in the above double sum $\sum_{k<k'}$ in $k$ and $k'$, the terms corresponding to $k'=k+1$ contribute to $\pm c_1 \kappa F_\pm$ (depending on $k\in \mathcal K_\pm$) and the other terms (\emph{i.e.} $k'\geq k+2$) only contribute to the error term (see also \eqref{eq:exa}).

Gathering all the above estimates, we have proved~\eqref{eq:E:R}.
\end{proof}

Combining~\eqref{eq:E:R} with \eqref{eq:energy} and~\eqref{eq:energlimit}, we obtain the following result.
\begin{corollary}\label{cor:ener}
Let $\vec{u}$ be a global solution of~\eqref{nlkg} satisfying the decomposition given by Lemma~\ref{le:dec} on $[t,\infty)$, for some $t\geq 0$. Let $1<\theta<\min (p-1,\frac{5}{4})$. Then
\begin{equation}\label{suite}
2 \alpha \int_{t}^{\infty}\|\partial_{t}u(s)\|_{L^{2}}^{2}\ud s
= -c_{1}\kappa F_+(t)+c_{1}\kappa F_-(t) +O(\mathcal{N}^2(t)+F^{\theta}_{+}(t)+F^{\theta}_{-}(t)).
\end{equation}
\end{corollary}

\subsection{General estimates for global solutions}
The following result is similar to Proposition~3.1 in \cite{CMYZ}.
We repeat the proof for the sake of completeness.

\begin{proposition}\label{pr:unif}
There exists a universal constant $\delta_{1} >0$ such that the following holds. Let $0<\delta<\delta_{1}$ and $\vec u$ be a global solution of \eqref{nlkg} satisfying \eqref{eq:decompo} with $K\geq 1$. Let $T_{\delta}\gg 1$ be such that $\vec u$ admits a decomposition as in Lemma~\ref{le:dec} in a neighborhood of $T_\delta$, with 
\begin{equation}\label{at:Td}
[\mathcal{N}(T_{\delta})]^2+b(T_\delta)+ F_{+}(T_{\delta})+F_{-}(T_{\delta})\le \delta^{2}.
\end{equation}
Then, for all $t\geq T_\delta$, it holds
\begin{equation}\label{eq:unif}
[\mathcal{N}(t)]^2+ b(t)+ F_{+}(t)+ F_{-}(t)\lesssim \delta^{2},
\end{equation}
and
\begin{equation}\label{est:F-}
F_{-}(t)\leq 3 \left(\delta^{-2}+\lambda (t-T_{\delta})\right)^{-1}.
\end{equation}
\end{proposition}
\begin{proof}
For a constant $C>1$ to be taken large enough, we introduce the following bootstrap estimates
\begin{equation}\label{BS:1}
\mathcal{N}\le C\delta,\quad b\le C\delta^{2},\quad F_{+}\le \delta^{\frac{3}{2}},\quad F_{-}\le \delta^{\frac{3}{2}},
\end{equation}
and we set 
\begin{equation*}
\TS=\sup\big\{t\in[T_{\delta},\infty) \text{ such that \eqref{BS:1} holds on $[T_{\delta},t]$} \big\}>T_\delta.
\end{equation*}
We prove $\TS=\infty$ by
strictly improving the bootstrap estimate~\eqref{BS:1} on $[\TD,\TS)$ for $C>1$ large enough.
In the rest of the proof, the implicit constants in $\lesssim$ or $O(\cdot)$ do not depend on the constant $C$ of the bootstrap estimate~\eqref{BS:1}. 

\emph{Estimate on $\mathcal N$.}
From~\eqref{eq:damped} and \eqref{BS:1}, it holds on $[\TD,\TS)$,
\begin{equation*}
\frac \ud{\ud t} \left[ e^{2\mu t} \mathcal F\right] \lesssim C^3 \delta^3 e^{2\mu t} + C \delta^{\frac 5 2} e^{2\mu t} \leq
 \delta^2 e^{2\mu t},
 \end{equation*}
for $\delta>0$ small enough (depending on $C$). From~\eqref{eq:new2} and \eqref{at:Td}, $\mathcal F(\TD)\lesssim \delta^2$. Thus, integrating the above estimate on $[\TD,t]$,
for any $t\in [\TD,\TS)$, $\mathcal F(t)\lesssim \delta^2$.
In particular, by \eqref{eq:new2}, we obtain
\begin{equation*}
\|\vve\|_\ENE^2 \lesssim \mathcal F+b
\lesssim C \delta^2.
\end{equation*}
Arguing similarly for the quantity $\BB$, we have
\begin{equation}\label{bnf3}
\sum_{k=1}^{K} |\ell_k|^2 + \sum_{k=1}^{K} (a_k^-)^2\lesssim \BB \lesssim \delta^{2}.
\end{equation}
Hence we obtain, for all $t \in [\TD, \TS)$,
\begin{equation}\label{est:N}
\mathcal N(t) \lesssim \sqrt{C} \delta.
\end{equation}
For $C$ large enough, this strictly improves the bootstrap estimate \eqref{BS:1} on~$\mathcal N$ on the interval $[\TD,\TS)$.

\emph{Estimate on $b$.}
Now, we prove that for $C$ large enough, it holds for all $t \in [\TD,\TS)$, 
\begin{equation}\label{eq:bbb}
b(t) \le \frac C2 \delta^2.
\end{equation}
From \eqref{suite} in Corollary~\ref{cor:ener} and \eqref{at:Td}, we have
\begin{equation}\label{int1.3}
\int_{\TD}^{\infty} \|\partial_{t} u(s)\|_{L^2}^2 \ud s \lesssim \mathcal N^2(\TD)+ F_+(\TD)+F_-(T_\delta)\lesssim \delta^2.
\end{equation}
By \eqref{at:Td}, we have $b(\TD)\leq \delta^2$.
For the sake of contradiction, take $C>4$ large and assume that there exists $t_2\in [\TD,\TS)$ such that
\begin{equation*}
b(\Td)= \frac C2 \delta^2,\quad b(t)< \frac C2 \delta^2 \quad \text{on } [\TD,t_{2}).
\end{equation*}
On the one hand, by continuity of $b$, there exists $\Tu\in [\TD,t_{2})$ such that
\begin{equation*}
b(\Tu)=\frac C4 \delta^2\quad \text{and}\quad
b(t)> \frac C4 \delta^2 \quad \text{on } (\Tu,t_{2}].
\end{equation*}
Using~\eqref{eq:b} and the bootstrap estimates~\eqref{BS:1}, we have
\begin{equation*}
\frac \ud{\ud t} b = 2\nu^+ b + O(C^3 \delta^3 + C \delta^{\frac 52})
\end{equation*}
which implies (for $\delta$ small enough depending on $C$)
\begin{equation}\label{bnf5}
\Td-\Tu=\frac{\log 2}{2\nu^+}+O(\delta^{\frac{1}{3}}),
\end{equation}
and thus
\begin{equation}\label{bnf6}
\int_{\Tu}^{\Td} b(s) \ud s \gtrsim C \delta^2.
\end{equation}

On the other hand, by \eqref{def:gs}, \eqref{def:ee},~\eqref{bnf3},~\eqref{int1.3} and \eqref{bnf5},
\begin{equation}\label{bnf7}
\int_{\Tu}^{\Td} \|\eta(t)\|_{L^2}^2 \ud t \lesssim \int_{\Tu}^{\Td} \bigg(\|\partial_{t}u(t)\|_{L^{2}}^{2}+\sum_{k=1}^{K}\|\ell_{k}\partial_{x}Q_{k}(t)\|_{L^{2}}^{2}\bigg)\ud t\lesssim\delta^{2}.
\end{equation}
By the definition of $a_k^\pm$, one has
\begin{equation*}
a_k^+=\zeta^+\langle \e,Y_k\rangle+\langle \eta,Y_k\rangle, \quad
a_k^-=\zeta^-\langle \e,Y_k\rangle+\langle \eta,Y_k\rangle
\end{equation*}
and thus
\begin{equation*}
a_k^+ = \frac{\zeta^+}{\zeta^-} a_k^- + \frac{\zeta^- -\zeta^+}{\zeta^-} \langle \eta,Y_k\rangle.
\end{equation*}
Combining \eqref{bnf3}, \eqref{bnf5} and \eqref{bnf7}, we find the estimate
$ \int_{\Tu}^{\Td} b(s) \ud s\lesssim \delta^{2}$,
which contradicts~\eqref{bnf6} for $C$ large enough. This proves \eqref{eq:bbb}.

\emph{Estimate on $F_{-}$.}
Let $1<\theta<\min (p-1,\frac{5}{4})$.
First, from~\eqref{suite} in Corollary~\ref{cor:ener} and~\eqref{BS:1}, we have
\begin{equation}\label{est:F+F-}
F_{+}\le 2F_{-}+O(\mathcal{N}^{2})\quad \mbox{and}\quad
F_{+}^\theta\le 2F_{-}^\theta+\mathcal{N}^{2}.
\end{equation}
Second, from~\eqref{eq:RF-},~\eqref{BS:1} and~\eqref{est:F+F-}, we have
\begin{align*}
\frac{\ud}{\ud t} R_- 
&\geq \left(\lambda+\frac{2}{\lambda F_{-}}\left(\mathcal{N}^{2}-F_{+}^{\theta}-F_{-}^{\theta}\right)\right)\exp\left( -3\lambda^{-1}\mathcal M \right)\\
&\geq \left(\lambda-\frac 6\lambda F^{\theta-1}_{-}\right)\exp\left( -3\lambda^{-1}\mathcal M \right)
\geq \frac{\lambda}{2}.
\end{align*}
By integration on $[\TD,t]$ for any $t\in [\TD,\TS)$, it holds 
\begin{equation}\label{est:R-bis}
R_-(t)\geq R_-(T_{\delta})+\frac{\lambda}{2}(t-T_{\delta}).
\end{equation}
By the definition of $R_-$ and~\eqref{at:Td}, we have $R_-(\TD) \geq \frac 12 \delta^{-2}$. 
Thus, \eqref{est:R-bis} implies that, for any $t\in [\TD,\TS)$
\begin{equation*}
F_{-}(t)\leq \frac 3{2 R_-}\leq \frac{3}{2} \left(\frac 12\delta^{-2}+\frac \lambda 2(t-T_{\delta})\right)^{-1}\le 3\delta^{2}.
\end{equation*}
This is strictly improves the bootstrap estimate~\eqref{BS:1} of $F_{-}$ on the interval $[\TD,\TS)$
and proves \eqref{est:F-}.

\emph{Estimate on $F_{+}$.}
From~\eqref{est:N},~\eqref{est:F+F-} and~\eqref{est:F-}, we observe that
\begin{equation*}
F_{+}\leq 2F_{-}+O(\mathcal{N}^{2})\leq O(C\delta^{2}).
\end{equation*}
For $\delta$ small enough (depending on $C$), this strictly improves the estimate \eqref{BS:1} of~$F_{+}$ on $[\TD,\TS)$.

The previous estimates prove that $\TS=\infty$ and that \eqref{eq:unif} holds on $[\TD,\infty)$.
\end{proof}

\section{Alternate signs property for neighbor solitary waves}
In this Section, we prove the following property.

\begin{proposition} \label{prop:alt_sign}
Let $\vec u$ be a global solution of \eqref{nlkg} such that $K\geq 2$ in \eqref{eq:decompo} of Theorem~\ref{th:1}.
Then, 
\begin{equation}\label{eq:alt}
\sigma_k=-\sigma_{k+1} \quad \mbox{for all $k\in \{1,\ldots,K-1\}$.}
\end{equation}
\end{proposition}
\begin{proof} We perform computations in the context of Proposition~\ref{pr:unif}. Assuming that
$\mathcal K_+$ is not empty, we reach a contradiction.
In this context, if $\mathcal K_-$ is empty, we have $F_-=0$. Otherwise, we may use the estimate~\eqref{est:F-} on $F_-$ proved in Proposition~\ref{pr:unif}.
Using~\eqref{eq:RF+},~\eqref{eq:unif} and~\eqref{est:F-}, we obtain that for any $t\in [T_{\delta},\infty)$,
\begin{equation}\label{est:dR+2}
\frac{\ud}{\ud t} R_+
\leq -\frac{3}{4}\lambda 
+\frac 3\lambda F_+^{\theta-1}
+\frac{2}{\lambda}R_+ F_{-}^{\theta}
\leq -\frac{\lambda}{2}+\frac {18}{\lambda} (\delta^{-2}+\lambda(t-T_{\delta}))^{-\theta}R_+.
\end{equation}
Define the auxiliary function
\begin{equation*}
\widetilde{R}^{+}(t)=R_+(t)\exp\left[\frac {18}{\lambda^2(\theta-1)}(\delta^{-2}+\lambda(t-T_{\delta}))^{1-\theta}\right],\quad \mbox{for $t\in [T_{\delta},\infty)$.}
\end{equation*}
By direct computation and~\eqref{est:dR+2}, we observe that
\begin{equation*}
\frac{\ud}{\ud t}\widetilde{R}^{+}\leq -\frac{\lambda}{2}
\exp\left[\frac {18}{\lambda^2(\theta-1)}(\delta^{-2}+\lambda(t-T_{\delta}))^{1-\theta}\right]
\leq -\frac{\lambda}{2}
.\end{equation*}
By integrating the above estimate on $[\TD,t]$, we obtain
\begin{equation*}
\widetilde{R}^{+}(t)
\leq \widetilde{R}^{+}(T_{\delta})-\frac{\lambda}{2}(t-\TD),
\end{equation*}
which is contradictory with $\widetilde{R}^{+}(t)\geq 0$ for large $t$.
This means that $\mathcal K_+=\emptyset$ and so $\mathcal K_-=\{1,\ldots,K-1\}$: the signs of the solitary waves alternate.
\end{proof}

\section{Description of long-time asymptotics} \label{S:5}

We consider any global solution $\vec u$ of \eqref{nlkg}. 
We prove Theorem~\ref{th:desc} by considering separately the no soliton case, and then the cases $K=1$ and $K\geq 2$ in Theorem~\ref{th:1}.

\subsection{No soliton case}
If $\vec{u}(t)$ converges to $0$ 
as $t\to \infty$, as a consequence of (7) of Theorem~2.3 in \cite{BRS}, $\vec u$ converges exponentially to $0$ in $H^1\times L^2$. Alternatively, one can use the energy functional
\begin{equation*} 
\int \big\{ (\partial_x u)^2+ (1-\rho\mu )u^{2}
+(\partial_t u+\mu u)^2\big\}
\end{equation*}
where $\mu>0$ is small and $\rho=2\alpha-\mu$
as in \S\ref{S.3.3} to prove the exponential convergence.

\subsection{Single soliton case}
Assume that $\vec u$ follows the solitary wave scenario in Theorem~\ref{th:1} with $K=1$.
Let $\delta>0$ to be chosen small enough.
Following Lemma~\ref{le:dec} and Proposition \ref{pr:unif}, there exists $T_\delta >0$ such that estimates \eqref{eq:unif}-\eqref{est:F-} hold on $[ T_\delta,\infty)$. 
When $K=1$, by convention
\begin{equation*}
F_{+}(t)=F_{-}(t)=0,\quad \mbox{for $t\ge 0$},
\end{equation*}
and the estimates proved in the previous Sections simplify; see Remarks~\ref{rk:KK01}, \ref{rk:K01} and \ref{rk:KK}.
In particular, from~\eqref{eq:b} and~\eqref{eq:damped}, it holds
\begin{equation}\label{eq:b:K=1}
|\dot{b}-2\nu^{+}b|\lesssim \mathcal{N}^{3},
\end{equation}
\begin{equation}\label{eq:damped:K=1}
\frac \ud{\ud t}\mathcal F + 2 \mu \mathcal F \lesssim \mathcal N^3, \quad \frac \ud{\ud t}\BB + 2 \mu \BB \lesssim \mathcal N^3,
\end{equation}
and from~\eqref{eq:new2} and~\eqref{BS:1},
\begin{equation}\label{est:N3:K=1}
\mathcal{N}^{3}\le C\delta\mathcal{N}^{2}\lesssim C\delta\left(\mathcal{F}+\frac{b}{2\mu}\right).
\end{equation}
Set $\tilde{b}=b-\delta^{\frac{1}{2}}\mathcal{F}$, observe that 
\begin{equation*}
\tilde{b}=\bigg(1+\frac{\delta^{\frac{1}{2}}}{2\mu}\bigg)b-\delta^{\frac{1}{2}}\bigg(\mathcal{F}+\frac{b}{2\mu}\bigg)\le 
\bigg(1+\frac{\delta^{\frac{1}{2}}}{2\mu}\bigg)b.
\end{equation*}
Therefore, using~\eqref{eq:b:K=1},~\eqref{eq:damped:K=1} and~\eqref{est:N3:K=1},
\begin{align*}
\frac{\ud}{\ud t}\tilde{b}
&\ge 2\nu^{+}b+2\mu \delta^{\frac{1}{2}}\mathcal{F}-C\delta^{\frac{3}{4}}\bigg(\mathcal{F}+\frac{b}{2\mu}\bigg)\\
&\ge (2\nu^{+}-\delta^{\frac{1}{2}})b+(2\mu \delta^{\frac{1}{2}}-C\delta^{\frac{3}{4}})\bigg(\mathcal{F}+\frac{b}{2\mu}\bigg)\ge \nu^{+}\tilde{b},
\end{align*}
where $\delta>0$ small enough such that 
\begin{equation*}
(2\nu^{+}-\delta^{\frac{1}{2}})\bigg(1+\frac{\delta^{\frac{1}{2}}}{2\mu}\bigg)\ge \nu^{+},\quad 
2\mu\delta^{\frac{1}{2}}-C\delta^{\frac{3}{4}}>0.
\end{equation*}
Integrating on $[t,s]\in [T_{\delta},\infty)$,
\begin{equation}\label{est:tiledb}
\tilde{b}(t)\le e^{-\nu^{+}(s-t)}\tilde{b}(s).
\end{equation}
Let $s\to \infty$ in~\eqref{est:tiledb} and using~\eqref{BS:1}, we obtain
\begin{equation*}
\tilde{b}(t)\le 0,\quad b(t)\le \delta^{\frac{1}{2}}\mathcal{F}(t) \quad \mbox{for $t\geq T_{\delta}$}.
\end{equation*}
Thus, using~\eqref{eq:b:K=1},~\eqref{eq:damped:K=1} and~\eqref{est:N3:K=1} again, 
\begin{equation*}
\begin{aligned}
\frac{\ud}{\ud t}\bigg(\mathcal{F}+\frac{b}{2\mu}\bigg)
&\le -2\mu \mathcal{F}+\frac{\nu^{+}}{\mu}b+C\delta^{\frac{3}{4}}\bigg(\mathcal{F}+\frac{b}{2\mu}\bigg)\\
&\le -(2\mu-C\delta^{\frac{3}{4}})\bigg(\mathcal{F}+\frac{b}{2\mu}\bigg)+\bigg(1+\frac{\nu^{+}}{\mu}\bigg)b
\le -\mu \bigg(\mathcal{F}+\frac{b}{2\mu}\bigg),
\end{aligned}
\end{equation*}
by possibly choosing $\delta>0$ small enough. Integrating on $ [T_{\delta},t]$, we obtain
\begin{equation*}
\bigg(\mathcal{F}+\frac{b}{2\mu}\bigg)(t)
\le e^{-\mu t}e^{\mu T_{\delta}}\bigg(\mathcal{F}+\frac{b}{2\mu}\bigg)(T_{\delta})\lesssim e^{-\mu t}e^{\mu T_{\delta}}\delta^{2}.
\end{equation*}
Therefore, using again \eqref{eq:new2},
\begin{equation}\label{est:N2:K=1}
\mathcal{N}^{2}(t)\lesssim \bigg(\mathcal{F}+\frac{b}{2\mu}\bigg)(t)\lesssim e^{-\mu t}.
\end{equation}
From~\eqref{eq:z} and \eqref{est:N2:K=1}, we have $|\dot z_1|\leq e^{-\frac \mu 2 t}$,
which proves that $z_1(t)$ converges exponentially to its limit as $t\to \infty$.
In view of the decomposition of $\vec u$ in \eqref{def:ee}, the proof in the case $K=1$ is complete.

\subsection{Multi-soliton case}
Assume that $\vec u$ follows the multi-solitary wave scenario in Theorem~\ref{th:1} with $K\geq 2$.
Let $\delta>0$ to be chosen small enough.
Following Lemma~\ref{le:dec} and Proposition \ref{pr:unif}, there exists $T_\delta >0$ such that estimates \eqref{eq:unif}-\eqref{est:F-} hold on $[ T_\delta,\infty)$. 
Recall that from Proposition \ref{prop:alt_sign}, the set $\mathcal K_+$ is empty.
In particular, following Remark~\ref{rk:KK}, we  use the estimates of Lemma~\ref{le:new}
ignoring the quantity $F_+$.
We start by showing that the quantity $\mathcal N(t)$ decays as $t^{-1}$.

\begin{proposition} \label{prop:N_decay}
There exists $T >0 $ such that the decomposition of $\vec u$ satisfies,
for all $t\geq T$
\begin{equation} \label{eq:strong}
F_-(t) \lesssim t^{-1}, \quad \mathcal N(t) \lesssim t^{-1}.
\end{equation}
\end{proposition}

\begin{proof}
The proof is inspired by that of \cite[Proposition 3.2]{CMYZ}. Let $0 < \delta < \delta_1$ in the context of Proposition \ref{pr:unif}. From \eqref{eq:unif} and \eqref{est:F-}, there exists $T>0$ large enough (fix any $T \ge 4T_\delta$) such that, for all $t\geq T/2$,
\begin{equation*}
\quad \mathcal N(t) \leq C\delta,\quad F_-(t) \leq \frac{6}{\lambda }t^{-1}.
\end{equation*}
 In particular, from~\eqref{eq:b} and~\eqref{eq:damped}
\begin{equation}\label{bnf20}
\left| \frac{\ud b}{\ud t} - 2\nu^+ b\right|\lesssim \mathcal N^3 + t^{-1}\mathcal N,\quad
\frac{\ud}{\ud t}\mathcal F+2\mu \mathcal F\lesssim\mathcal N^3 + t^{-1}\mathcal N.
\end{equation}
Our goal is to obtain the decay rate of $\mathcal N$. The above bounds are not quite enough because of the  term $\mathcal N^3$ for which only smallness is known at this point. This is the reason why we will work on a modification $\tilde b$ of $b$.
Recall~\eqref{eq:new2}: 
\[ 0 \le \mathcal N^2 \lesssim \mathcal F+\frac{b}{2\mu}. \]
For $0<\omega \ll 1$ to be chosen later, observe that, for $\delta$ small,
\begin{equation}\label{bnf21}
\mathcal N^3 + t^{-1} \mathcal N \lesssim 
\omega^2 \mathcal N^2+ \omega^{-2} t^{-2}\lesssim
\omega^2\left(\mathcal F+\frac{b}{2\mu} \right)+ \omega^{-2}t^{-2}.
\end{equation}
(Here and below the implied constants do not depend on $\omega$). Set $\tilde b=b- \omega \mathcal F$ and observe that
\begin{equation*}
\tilde b=b- \omega \mathcal F
=\left(1+\frac{\omega}{2\mu}\right) b - \omega\left(\mathcal F+\frac{1}{2\mu} b\right)
\leq \left(1+\frac {\omega} {2\mu}\right) b .
\end{equation*}
Therefore, using~\eqref{bnf20} and~\eqref{bnf21},
\begin{align*}
\frac{\ud \tilde b}{\ud t}
&\geq 2\nu^+ b + 2\omega \mu \mathcal F - C \omega^2\left(\mathcal F+\frac{b}{2\mu}\right)
-C\omega^{-2} t^{-2}\\
&\geq (2\nu^+-\omega) b
+\left( 2 \omega \mu - C \omega^2 \right)\left(\mathcal F+\frac{b}{2\mu}\right)
-C\omega^{-2} t^{-2}\\
& \geq \nu^+ \tilde b -C\omega^{-2} t^{-2},\end{align*}
where $\omega>0$ is taken small enough such that 
\begin{equation*}
(2\nu^+-\omega)\left(1+\frac {\omega} {2\mu}\right)^{-1}\geq \nu^+,\quad
2 \omega \mu - C \omega^2 >0.
\end{equation*}
Integrating on $[t,s] \subset [T/2,\infty)$, we obtain
\begin{equation}\label{est:tb}
\tilde b(t)- e^{-\nu^+ (s-t)} \tilde b(s)
\le C\omega^{-2} \int_{t}^{s} e^{-\nu^+(\tau-t)} \tau^{-2} \ud \tau
\lesssim \omega^{-2} t ^{-2}.
\end{equation}
Let $s \to \infty$ in~\eqref{est:tb} and using~\eqref{BS:1}, we obtain for all $t\geq T/2$,
\begin{equation*}
 \tilde b(t) \lesssim \omega^{-2} t^{-2},\quad 
 b(t) \leq C \omega^{-2} t^{-2}+\omega \mathcal F(t).
\end{equation*}
Thus using~\eqref{bnf20} and~\eqref{bnf21} again, it holds
\begin{align*}
\frac{\ud}{\ud t} \left(\mathcal F+\frac{b}{2\mu}\right)
&\leq - 2\mu \mathcal F + \frac{\nu^+}{\mu} b + C\omega^2 \left(\mathcal F+\frac{b}{2\mu}\right)+ C\omega^{-2}t^{-2}\\
&\leq - (2\mu-C\omega^2) \left(\mathcal F+\frac{b}{2\mu}\right) + \left(1+\frac{\nu^+}{\mu} \right) b + C\omega^{-2}t^{-2}
\\
&\leq - \mu \left(\mathcal F+\frac{b}{2\mu}\right) + C\omega^{-2}t^{-2},
\end{align*}
by possibly choosing $\omega>0$ small enough.
Integrating on $[T/2,t]$, we obtain
\begin{equation*}
\left(\mathcal F+\frac{b}{2\mu}\right)(t)
- e^{- \mu (t-\frac{T}{2})} \left(\mathcal F+\frac{b}{2\mu}\right)(T/2)
\lesssim \int_{\frac{T}{2}}^t e^{- \mu (t-s)} s^{-2}\ud s
\end{equation*}
Therefore, using again~\eqref{eq:new2},
\begin{equation*}
\mathcal N^2 (t)\lesssim \left(\mathcal F+\frac{b}{2\mu}\right)(t)
\lesssim t^{-2} + \delta^2 e^{\mu T/2} e^{-\mu t}\le t^{-2},
\end{equation*}
which proves \eqref{eq:strong}.
\end{proof}

We continue the proof of Theorem \ref{th:desc}. In view of the alternate signs property~\eqref{eq:alt} and the decay estimate \eqref{eq:strong}, the system \eqref{eq:y_k} rewrites as, 
for $k=2,\cdots,K-1$, and any $t \ge T$,
\begin{equation}\label{eq:edo}
\left\{ \begin{aligned}
&\dot{y}_{1}=-\frac{\kappa}{2\alpha}e^{-(y_2-y_1)} + O \left( t^{-\theta} \right),\\
& \dot y_k = \frac{\kappa}{2\alpha} \left( e^{-(y_k - y_{k-1})} - e^{-(y_{k+1} - y_{k})} \right) + O \left( t^{-\theta} \right),\\
&\dot{y}_{K}=\frac{\kappa}{2\alpha}e^{-(y_{K}-y_{K-1})}+O(t^{-\theta}).
\end{aligned}\right.
\end{equation}
This system of ODEs is studied in \cite{MZ} and \cite{CZ}, where it appears naturally in a different context (the description of characteristic blowup points of the semilinear wave equation),  with  slightly different perturbation terms. For the convenience of the reader, we provide a study of the dynamics.

We introduce an explicit solution to the unperturbed ODE system~\eqref{eq:uedo}
\begin{equation} \label{def:y_k}
\bar y_k(t) := \left( k - \frac{K+1}{2} \right) \log t + \tau_k, \quad \hbox{for $k=1, \dots, K,$}
\end{equation}
where $(\tau_k)_{k=1, \dots, K}$ are constants uniquely defined by~\eqref{def:tau_k}
(see also \cite[p. 1549]{CZ}). 

From~\eqref{eq:edo}, we observe that $\sum_{k=1}^K \dot y_k = O \left( t^{-\theta} \right)$. Since  $\theta>1$, there exist $\bar y_\infty \in \mathbb R$ such that
\begin{equation} \label{def:y_infty}
\frac{1}{K} \sum_{k=1}^K y_k = \bar y_\infty + O \left( t^{-\theta+1} \right).
\end{equation}
We introduce, for $k=1,\dots, K$,
\[ \xi_k(t) : = y_k - \bar y_k - \bar y_\infty. \]
Observe that, for $k=1,\cdots,K-1$,
\begin{equation*}
 y_{k+1} - y_{k}=\xi_{k+1} - \xi_k+\log t + (\tau_{k+1} - \tau_k),
\end{equation*} 
 so that using~\eqref{def:tau_k},
\[ \frac{\kappa}{2\alpha} e^{- (y_{k+1} - y_k)} = \frac{\kappa}{2\alpha} e^{ - (\xi_{k+1} - \xi_k ) - \log t - (\tau_{k+1} - \tau_k)} = \gamma_k t^{-1}{e^{ - (\xi_{k+1} - \xi_k )}}. \]
Therefore, from~\eqref{eq:edo} and $\gamma_{k} - \gamma_{k-1} = \frac{K+1}{2} -k$, it holds
($2\le k\le K-1$)
\begin{equation}\label{eq:edoxi}
\left\{ \begin{aligned}
&\dot{\xi}_{1}=-t^{-1}\gamma_{1}\big(e^{-(\xi_{2}-\xi_{1})}-1\big) + O \left( t^{-\theta} \right),\\
&\dot \xi_k= t^{-1} \left( \gamma_{k-1} (e^{ - (\xi_{k} - \xi_{k-1} )}-1) - \gamma_k (e^{ - (\xi_{k+1} - \xi_k )}-1) \right) + O \left( t^{-\theta} \right),\\
&\dot{\xi}_{K}=t^{-1}\gamma_{K-1} (e^{ - (\xi_{K} - \xi_{K-1} )}-1)+O(t^{-\theta}).
\end{aligned}\right.
\end{equation}
To complete the proof of Theorem \ref{th:desc}, it suffices to show that
for any $ k=1, \dots, K$,
\begin{align} \label{eq:th} 
\xi_k(t) = O(t^{-\theta+1}).
\end{align}
First, we prove a bound on $\xi_k$.
\begin{lemma}
There exists $M >0$ such that for all $k=1, \dots, K$ and for all $t\geq T_\delta$,
it holds $|\xi_k(t)| \le M$.
\end{lemma}

\begin{proof}
Set
\begin{equation*}
\zeta_{0}=\zeta_{K}=0\quad \mbox{and}\quad \zeta_k = \xi_{k+1} - \xi_k\quad \mbox{for}\ k=1,\cdots,K-1.
\end{equation*}
From~\eqref{eq:edoxi}, it holds for $k=1, \dots, K-1$ (with $\gamma_{0}=\gamma_{K}=0$) 
\begin{equation}\label{eq:zeta} 
\dot \zeta_k 
 = t^{-1} \left( - \gamma_{k+1} (e^{ - \zeta_{k+1}} -1) + 2 \gamma_k (e^{- \zeta_k}-1) - \gamma_{k-1} (e^{-\zeta_{k-1}} -1) \right)  + O\left( t^{-\theta} \right) .
\end{equation}

\begin{claim}\label{cl:theclam}
There exists $M_1>0$ such that for all $k=1, \dots, K-1$, and for all $t \ge T_\delta$, it holds $|\zeta_k(t)| \leq M_1$.
\end{claim}

\begin{proof}[Proof Claim~\ref{cl:theclam}]
First, we prove a lower bound $\zeta_k(t) \geq -M_1$, for some $M_1>0$.
Fix $\vartheta = \frac{1+\theta}{2}$ and for $D_1>0$ to be fixed later, denote for $k=1,\cdots,K-1$,
\begin{equation*}
\rho_k = \gamma_k (e^{- \zeta_k}-1) + {D_1 t^{-\vartheta+1}}, \quad \rho_+ = \max_k \rho_k.
\end{equation*}
Let us prove that $\rho_+$ is non increasing for large enough times
using a bootstrap argument. For $C_1>0$ and $T_1 \ge T_\delta$ to be chosen later, let
\[ T_* = \sup \{ t \in [T_1,\infty) \text{ such that } \rho_+ \le C_1 \text{ on } [T_1,t] \}. \]

Let $t \in [T_1,T_*)$, and consider an index $k$ such that $\rho_k(t) = \rho_+(t)$. Observe that 
\begin{gather}
\gamma_{k}e^{-\zeta_{k}(t)}\le \rho_{k}(t)+\gamma_{k}\le C_{1}+K^{2},\label{est:gez}\\
2\gamma_k \big(e^{- \zeta_k(t)}-1\big) \ge \gamma_{k+1} \big(e^{- \zeta_{k+1}(t)}-1\big)+ \gamma_{k-1} \big(e^{- \zeta_{k-1}(t)}-1\big).\label{est:maxk}
\end{gather} 
Gathering \eqref{eq:zeta},~\eqref{est:gez} and~\eqref{est:maxk}, there exist $C_{0}>0$ such that 
at $t$,
\begin{align*} 
\dot \rho_k & = - \gamma_k e^{- \zeta_k} \dot \zeta_k - {D_1(\vartheta-1)}t^{-\vartheta} \\
& = - \gamma_k e^{- \zeta_k} t^{-1} \left( - \gamma_{k+1} (e^{ - \zeta_{k+1}} -1) + 2 \gamma_k (e^{- \zeta_k}-1) - \gamma_{k-1} (e^{-\zeta_{k-1}} -1) \right)\\
&\quad +O(\gamma_{k}e^{-\tau_{k}(t)}t^{-\theta}) - {D_1(\vartheta-1)}t^{-\vartheta} \\
& \le C_{0}\gamma_k e^{- \zeta_k(t)}t^{-\theta} - {D_1(\vartheta-1)}t^{-\vartheta} 
\\&\le \left( (C_1 + K^2) C_0 - D_1(\vartheta-1) T_1^{\theta - \vartheta} \right) t^{-\theta}.
\end{align*}
Fix $T_1 \ge T_\delta$ and then $D_1$, $C_1$ such that
\begin{gather*}
(\vartheta-1) T_1^{\theta - \vartheta} \ge C_0+1,\\
D_1 \ge C_0(\rho_+(T_1) +1 + K^2)+1,\quad C_1 = \rho_+(T_1) + D_1 +1.
\end{gather*}
Then, there holds
\begin{align*}
(C_1 + K^2) C_0 - D_1(\vartheta-1) T_1^{\theta - \vartheta} & \le (\rho_+(T_1) + D_1 +1 + K^2) C_0 - D_1 (C_0+1) \\
& \le (\rho_+(T_1) +1 + K^2)C_0 - D_1 \le -1.
\end{align*}
Since $\rho_+(T_1) < C_1$, by continuity, $T_* >T_1$. We also have $\dot \rho_k(t) <0$,
so that $\rho_k$ is decreasing at $t$; note that this property holds for any index $k$ such that $\rho_k(t) = \rho_+(t)$. If $j$ is an index such that $\rho_j(t) < \rho_+(t)$, then by continuity this inequality holds on a  neighborhood of $t$. Thus $\rho_+$ is decreasing at $t$, for any $t \in [T_1,T^*)$. In particular, $\rho_+(t) \le \rho_+(T_1) \le C_1$. By continuity, we obtain $T_* = \infty$ and so for all $t \ge T_1$, $\rho_+(t) \le C_1$.

By continuity, there exists $C_2$ such that for all $t \ge T_\delta$, $\rho_+(t) \le C_2$, and so for all $k=1,\dots, K-1$, for all $t \ge T_\delta$,
$e^{-\zeta_k(t)} \le 2C_2+1$,
since $D_1 >0$ and $\gamma_k \ge 1/2$. Therefore, for $M_1 = \log(2C_2+1)$, we have proved the lower bound on $\zeta_k(t)$.

Arguing similarly using the minimum of $\tilde \rho_k = \gamma_k (e^{- \zeta_k}-1) - {D_1}{t^{-\vartheta+1}}$, one also proves an upper bound on $\zeta_k$.\end{proof}

By~\eqref{def:y_infty} and
$\sum_{k=1}^K \bar y_k = \sum_{k=1}^K \tau_k =0,$
there exists $M_2>0$ such that for all $t \ge T_\delta$,
$\big| \sum_{k=1}^K \xi_k(t) \big| \le M_2$.
By contradiction, assume that for some $k_0$ and $t \ge T_\delta$ 
\[ \xi_{k_0}(t) \ge M \quad \text{where} \quad M := \frac{2M_2}{K} + \frac{K-1}{2} M_1. \]
Then,
\begin{gather*}
 \xi_k(t) = \xi_{k_0} (t) + \sum_{j=k_0}^{k-1} \zeta_j \ge M - |k-k_0| M_1\quad \mbox{for}\ 1\le k_{0}<k\le K,
\\
\xi_k(t) = \xi_{k_0} (t) - \sum_{j=k}^{k_{0}-1} \zeta_j \ge M - |k-k_0| M_1\quad \mbox{for}\ 1\le k< k_{0}\le K,
\end{gather*}
so that (in view of the definition of $M$)
\[ \sum_{k=1}^K \xi_k(t) \ge KM - M_1 \sum_{k=1}^K |k-k_0| \ge K \left( M - \frac{K-1}{2} M_1\right) \ge 2 M_2, \]
which is a contradiction. Therefore, for all $k=1, \dots, K$ and $t \ge T_\delta$,
$\xi_k(t) \le M$.
One argues similarly to show that $ \xi_k(t) \ge -M$.\end{proof}

Now, we consider the unperturbed ODE system for the $(\xi_k)_{k=1,\ldots,K}$, that is
\begin{align} \label{def:ODE}
\dot {\boldsymbol \varpi} = t^{-1} \Phi({\boldsymbol \varpi}),
\end{align}
where ${\boldsymbol \varpi} = (\varpi_k)_{k=1,\dots, K}$ and $\Phi: \mathbb R^K \to \mathbb R^K$ is defined by
\begin{equation*}
\left\{ \begin{aligned}
&\Phi_{1}(\boldsymbol \varpi)= - \gamma_1 (e^{ - (\varpi_{2} - \varpi_1 )}-1),\\
&\Phi_k(\boldsymbol \varpi) = \gamma_{k-1} (e^{ - (\varpi_{k} - \varpi_{k-1} )}-1) - \gamma_k (e^{ - (\varpi_{k+1} - \varpi_k )}-1),\ \mbox{for}\ 2\le k\le K-1,\\
&\Phi_K(\boldsymbol \varpi) = \gamma_{K-1} (e^{ - (\varpi_{K} - \varpi_{K-1} )}-1).
\end{aligned}\right.
\end{equation*}
This system is studied in \cite{CZ}. Observe that setting $\boldsymbol{e} _1 = \frac{1}{\sqrt{K}} (1,\dots, 1)^T$, it holds
\[ \mbox{for all }  t,t_0 \ge T_\delta, \quad \frac{1}{K} \sum_{j=1}^K \varpi_j(t) = \frac{1}{\sqrt{K}} (\boldsymbol \varpi, \boldsymbol{e} _1 ) = \frac{1}{K} \sum_{j=1}^K \varpi_j(t_0).\]
Moreover, $D\Phi(0)$ is the $K\times K$ matrix with entries
\begin{equation*}
\begin{aligned}
&m_{1,1}=-\gamma_{1},\quad m_{K,K}=-\gamma_{K-1},\quad m_{k,k}=-(\gamma_{k-1}+\gamma_{k}),\quad \mbox{for}\ k=2,\cdots,K-1,\\
&m_{k,k-1}=\gamma_{k-1},\quad m_{k,k+1}=\gamma_{k},\quad m_{k,k'}=0,\quad \mbox{if}\ |k-k'|\ge 2.
\end{aligned}
\end{equation*}
We recall the following properties.
\begin{proposition}[\cite{CZ}] \label{prop:Phi}
It holds $D\Phi(0)\boldsymbol{e}_1=0$ and
\begin{align} \label{coer:dPhi}
\mbox{for all } \boldsymbol{x} \in \boldsymbol e_1^\perp, \quad (D\Phi(0) \boldsymbol{x}, \boldsymbol{x}) \le - \| \boldsymbol x \|^2.
\end{align}
Furthermore, for any $M>0$, there exists $C(M)>0$ such that  for any $t_0 >0$, if $|{\boldsymbol \varpi} (t_0)| \le M$, then
\begin{align} \label{Phi:conv}
\mbox{for all } t \ge t_0, \quad \| \boldsymbol \varpi(t) - (\boldsymbol \varpi(t_0), \boldsymbol{e}_{1}) \boldsymbol{e}_1 \| \le C(M) {t_0}t^{-1}.
\end{align}
\end{proposition}

\begin{proof}
See \cite[Lemma 2.8]{CZ} for the coercivity \eqref{coer:dPhi} of $D\Phi(0)$ and \cite[Proposition~2.5]{CZ} for the convergence \eqref{Phi:conv} (written there in the variable $\tau = \log t$).
\end{proof}

Now, we complete the proof of  Theorem \ref{th:desc} by showing
the estimate~\eqref{eq:th}.

\begin{proof}[Proof of~\eqref{eq:th}]
We summarise what was obtained so far. Let ${\boldsymbol \xi} := (\xi_k)_{k =1, \dots, K}$. There exists $C_{\theta}$ such that for $t\ge T_{\delta}$,
\begin{enumerate}

\item \label{xi:1}$\displaystyle \dot {\boldsymbol \xi} = t^{-1} \Phi({\boldsymbol \xi}) +\boldsymbol r(t)$ where $\|\boldsymbol r(t) \| \le C_\theta t^{-\theta}$,

\item \label{xi:2} $ |(\boldsymbol \xi, \boldsymbol e_1) |\le C_{\theta} t^{-\theta+1} $,

\item $\| \boldsymbol \xi (t) \| \le \sqrt{K} M$, for $t\ge T_{\delta}$.
\end{enumerate}
Moreover, by \eqref{coer:dPhi}, there exist $\epsilon, C >0$ such that if $\|\boldsymbol x \| \le \epsilon$, then
\begin{align} \label{coer:Phi}
( \Phi(\boldsymbol x), \boldsymbol x ) \le - \frac{1}{2} \| \boldsymbol x \| ^2 + C (\boldsymbol x, \boldsymbol e_1)^2.
\end{align}
\emph{Step 1}. We claim that there exists $T_{\epsilon}\ge T_{\delta}$ such that
\begin{equation}\label{est:xi}
\mbox{for all $ t\ge T_{\epsilon}$,}\quad \|\boldsymbol \xi (t)\|\le \epsilon. 
\end{equation}
Let $L = {4}\epsilon^{-1} C(M)+1$, where $C(M)$ given by Proposition \ref{prop:Phi}.
Set
\begin{equation*}
A = 3+ \sup \{ \| D\Phi(\boldsymbol x) \| : \| \boldsymbol x \| \le \max(C(M),\sqrt{K}M) \}.
\end{equation*} 
By~\ref{xi:1} and~\ref{xi:2}, we fix $t_{0}$ large enough such that
\begin{equation*}
\sup_{t\ge t_{0}}\big(t\sup_{s\ge t}\| \boldsymbol r(s)\|\big)\le \frac{\epsilon}{4}L^{-A}\quad \mbox{and}\quad 
t_{0}^{-\theta+1}\le \frac{\epsilon}{2}C_{\theta}^{-1}.
\end{equation*}
For any $t_{1}\ge t_{0}$, we denote ${\boldsymbol \varpi}_{t_1}$ the solution of \eqref{def:ODE} with   data 
$\boldsymbol \varpi_{t_1}(t_1) = \boldsymbol \xi(t_1)$ at time~$t_1$.
On the one hand, by standard Gronwall estimates, it holds for any $t\ge t_{1}$,
\begin{equation*}
 \| \boldsymbol \xi(t)-\boldsymbol \varpi_{t_1}(t) \| \le t^{A}\int_{t_1}^t s^{-A} \|\boldsymbol r(s)\| \ud s \le {t^A}{t_1^{-A+1}} \sup_{s \ge t_1} \| \boldsymbol r(s) \|. 
\end{equation*}
Let $t=Lt_{1}$, using the definition of $t_{0}$, we obtain
\begin{equation}\label{est:xi-pi}
\| \boldsymbol \xi(Lt_{1})-\boldsymbol \varpi_{t_{1}}(Lt_{1}) \|
\le L^{A}\big(t_{1}\sup_{s \ge t_1}\|\boldsymbol r(s)\|\big)\le \frac{\epsilon}{4}.
\end{equation}
On the other hand, 
using~\eqref{Phi:conv},~\ref{xi:2} and the definition of $t_{0}$, 
\begin{equation}\label{est:pi}
\|\boldsymbol \varpi_{t_1}(L t_1) \|
\le \|\big(\boldsymbol \xi(t_1),\boldsymbol e_1\big)\boldsymbol e_1\|+C(M)t_{1}(Lt_{1})^{-1}\le \frac{3}{4}\epsilon.
\end{equation}
Let $T_{\epsilon}=Lt_{0}$, from~\eqref{est:xi-pi} and~\eqref{est:pi}, we obtain~\eqref{est:xi}.

\emph{Step 2}.
Using~\ref{xi:1},~\ref{xi:2},~\eqref{coer:Phi} and~\eqref{est:xi}, we infer that for $t\ge T_{\epsilon}$,
\begin{align*}
\frac{d}{dt} \| \boldsymbol \xi \|^2 & = 2t^{-1} ( \Phi( \boldsymbol \xi), \boldsymbol \xi) + 2 (r, \boldsymbol \xi) \\
& \le - t^{-1} \| \boldsymbol \xi \|^2 + O( t^{1-2\theta}) + O( t^{-\theta} \| \boldsymbol \xi \|) \\
& \le - \left(\theta-\frac 12\right)t^{-1} \| \boldsymbol \xi \|^2 + O ( t^{1-2\theta}).
\end{align*}
A direct integration and $1<\theta <3/2$ yield, for some $C >0$,
\begin{equation*}
\|\boldsymbol \xi(t)\|^{2}\le \left(\frac{T_{\epsilon}}t\right)^{\theta-\frac 12}\|\boldsymbol \xi(T_{\epsilon})\|^{2}+Ct^{-2\theta+2}
\lesssim t^{-2\theta+2},
\end{equation*}
which is~\eqref{eq:th}.
\end{proof}

\section{Construction of multi-solitary waves}
In this Section, we prove Theorem~\ref{th:2}, adapting arguments from~\cite[Section 4]{CMYZ}.
Here, $\mathcal B_\ENE(\delta)$ denotes the open ball  of $H^1\times L^2$ of center $0$ and radius $\delta$ and $\mathcal B_{\RR^{K}}(\delta)$ (respectively, $\bar {\mathcal B}_{\RR^{K}}(\delta)$) denotes the open ball (respectively, closed ball)  of $\RR^K$ of center $0$ and radius $\delta$. Last, $\mathcal S_{\RR^K}( \delta)$ denotes the sphere of $H^1\times L^2$ of center $0$ and radius $\delta$.
We also use the notation from  \S\ref{S:2} and set
\begin{equation*}
\beta := \frac{1}{2 \sqrt{\alpha^2 + \nu_0^2}} = \langle \vec Y^+, \vec Z^+ \rangle^{-1} >0.
\end{equation*}

We recall the following preliminary result (for the proof, see \cite[Lemma~4.1]{CMYZ}).

\begin{lemma}\label{le:W}
Let $(z_{k},\ell_{k})_{k=1, \dots, K} \in \RR^{2K}$ be such that 
\[r = \min (z_{k+1} - z_k , k=1,\dots, K-1)\]
is large enough. There exist linear maps
\begin{equation*}
B:\RR^K \to \RR^K,\quad V:\RR^K\to \RR^K
\end{equation*}
smooth in $(z_{k},\ell_{k})_{k=1, \dots, K}$, satisfying
\begin{equation*}
\| B- \beta \Id\|\lesssim e^{-\frac 1 2 r}, \quad \|V_j\|\lesssim e^{-\frac 1 2 r},
\end{equation*}
and such that the function $W(\boldsymbol a ):\RR \to \RR$ defined by
\begin{equation*}
W(\boldsymbol a)(x) := \sum_{k=1}^K \bigg\{ B_k(\boldsymbol a) Y_k(x)+ V_{k} (\boldsymbol a ) \partial_{x} Q_k(x) \bigg\},
\end{equation*}
for any $\boldsymbol a = (a_1,\dots, a_k)$, 
satisfies, for all $k=1,\dots, K$, 
\begin{equation*}
\langle W(\boldsymbol a),\partial_{x} Q_k\rangle =0,\quad \langle W(\boldsymbol a),Y_k\rangle = \beta a_k.
\end{equation*}
In particular, setting
\begin{equation*}
\vec W(\boldsymbol a)=\begin{pmatrix} W(\boldsymbol a) \\ \nu^+ W(\boldsymbol a) \end{pmatrix} \quad\text{it holds}\quad
\langle \vec W(\boldsymbol a),\vec Z_k^+\rangle= a_k.
\end{equation*}
\end{lemma}


The next proposition and the invariance by translation of~\eqref{nlkg}
imply Theorem~\ref{th:2}.

\begin{proposition}\label{pr:multi}
Let $K\ge 2$ and $\sigma =\pm 1$. For $\delta>0$ small enough, let any
\begin{equation}\label{eq:pr:10}\left\{\begin{aligned}
&( \ell_k(0))_{k=1,\dots, K} \in \mathcal B_{\RR^{K}}(\delta),\\
& (z_k(0))_{k=1, \dots, K} \in \RR^{K} \\
&\qquad \text{satisfying } \min \{ z_{k+1}(0) - z_k(0), k=1, \dots, K-1 \} > 5 |\log \delta|, \\
& \vvep(0)\in \mathcal B_\ENE(\delta) \\&\qquad\text{satisfying } \eqref{ortho} \text{ and } \langle \vvep(0),\vec Z_k^+(0)\rangle=0 \text{ for } k=1,\dots, K.
\end{aligned}\right.\end{equation} 
There exists $\boldsymbol a_\sharp^+(0) = (a_{\sharp,k}^+(0))_{k=1,\dots, K} \in \bar{\mathcal B}_{\RR^K}(\delta^\frac 54)$ such that the solution $\vec{u}_\sharp$ of~\eqref{nlkg} with the initial data
\begin{equation*}
\vec u_\sharp(0)= \sigma \sum_{k=1}^K (-1)^k (Q(\cdot - z_k(0)), 0) + \vec W(\boldsymbol a_\sharp^+(0)) + \vvep(0)
\end{equation*}
is global and satisfies~\eqref{eq:th:2} where, for $k=1, \dots, K$,
$z_k = \bar y_k + y_\sharp + O(t^{-\theta+1})$,
for some $y_\sharp \in \RR$ and  $\bar y_k$ being defined in \eqref{def:y_k}.
\end{proposition}

\begin{remark}
From the proof of Proposition~\ref{pr:multi}, there exist even solutions of \eqref{nlkg} with any odd number $K\geq 3$ of
solitary waves.
\end{remark}

\begin{proof}
Given $\boldsymbol a^+(0) = (a_{k}^+(0))_{k=1,\dots, K} \in \bar{\mathcal B}_{\RR^K}(\delta^\frac 54)$, we consider the solution $\vec u(t)$ of \eqref{nlkg} with initial data
\[ \vec u(0)= \sigma \sum_{k=1}^K (-1)^k (Q(\cdot - z_k(0)), 0) + \vec W(\boldsymbol a^+(0)) + \vvep(0). \]

\emph{Decomposition.}
For any $t\geq 0$ such that $\vec u(t)$ is defined and satisfies~\eqref{for:dec}, we consider  its decomposition according to Lemma~\ref{le:dec}. Following \S\ref{S:2.4}, we introduce the notation $y_k$ ($k=1,\dots, K$), $\mathcal N$, $\mathcal M$ and
\[
F = \sum_{k=1}^{K-1} e^{ - (y_{k+1} - y_k)},
\quad R = \frac{1}{F} \exp ( - 3 \lambda^{-1} \mathcal M)
\]
(due to the choice of signs in the decomposition of $\vec u(0)$, $F = F_-$ and $F_+=0$). Note that by the properties of the function $W$ in Lemma~\ref{le:W} and the orthogonality properties \eqref{ortho} of $\vvep(0)$ assumed in~\eqref{eq:pr:10}, the initial data $\vec u(0)$ is modulated. Indeed, $(z_k(0),\ell_k(0))_{k=1,\dots, K}$ and
\begin{equation*}
\vve(0)=\vec W(\boldsymbol a^+(0)) + \vvep(0),
\end{equation*}
are the parameters of the decomposition of $\vec u(0)$.
 In particular, it holds from~\eqref{eq:pr:10}
\begin{equation*}
\mathcal N(0)\lesssim \delta,\quad F(0)\lesssim \delta^2.
\end{equation*}
Moreover, by Lemma~\ref{le:W}, for $k=1,\cdots, K$, it holds
\begin{equation*}
\langle \vve(0),\vec Z_k^+(0)\rangle = \langle \vec W(\boldsymbol a^+(0)), Z_k^+ \rangle = a_k^+(0),
\end{equation*}
which is consistent with the definition of $a_k^+$ in (v) of Lemma~\ref{le:dec}.

\emph{Bootstrap estimates.}
We introduce the following bootstrap estimates
\begin{equation}\label{BS:2}
\mathcal N\leq \delta^{\frac 34},\quad
F \leq \delta^{\frac 32},\quad
b\leq \delta^{\frac 52}
\end{equation}
and we set $\TS=\sup\left\{ t\in [0,\infty)\hbox{ such that \eqref{BS:2} holds on $[0,t]$}\right\}\geq 0$.

\emph{Estimates on the damped components.}
The estimate on $\mathcal N$ is strictly improved on $[0,\TS]$ as in the proof of Proposition~\ref{pr:unif}.
In particular, $\mathcal N\lesssim \delta$ on $[0,\TS]$.

\emph{Estimate on the distance.}
Inequality \eqref{eq:RF-} rewrites
\begin{equation*}
\frac{\ud R}{\ud t} \ge \left( \lambda + \frac{2}{\lambda F} (\mathcal N^2 - F^{\theta}) \right) \exp( - 3 \lambda^{-1} \mathcal M).
\end{equation*}
Now for $t \in [0,\TS]$, in view of \eqref{BS:2}, $\exp( - 3 \lambda^{-1} \mathcal M) = 1 + O(\mathcal N^2) = 1 + O(\delta^{3/2})$ and $F^{\theta-1} = O(\delta^{3(\theta-1)/2}) \le \frac{1}{6\lambda^2}$ so that
\[ \frac{\ud R}{\ud t} \ge \frac{2}{3} \lambda (1- C \delta^{3/2}) \ge \frac{\lambda}{2}. \]
Integrating on $[0,t]\subset [0,T_{*}]$, it holds
$R(t) \ge R(0) + \frac{\lambda}{2} t$.
Hence, as we also have $\exp( 3 \lambda^{-1} \mathcal M) =1 + O(\delta^{3/2})$ and $R(0)^{-1} \le 2F(0) \lesssim \delta^2$, 
\begin{equation*}
F = \frac{1}{R} \exp( 3 \lambda^{-1} \mathcal M) \le \frac{1}{R(0) + \frac{\lambda}{2} t} (1 + O(\delta^{3/2})) \lesssim \delta^{2}.
\end{equation*}
This strictly improves the estimate of $F$ in \eqref{BS:2}, and as in the proof of Proposition~\ref{prop:N_decay}, we also obtain the decay $F \lesssim t^{-1}$.
 
\emph{Transversality condition.}
From \eqref{eq:b} and $\mathcal N\lesssim \delta$, we observe that for any time $t\in [0,\TS]$ where it holds $b(t)=\delta^{\frac 52}$, we have
\begin{equation*}
\frac \ud{\ud t} b(t) \geq 2\nu^+ b(t) - C \delta^{3}
\geq 2\nu^+ \delta^{\frac 52} - C \delta^{3} \geq \nu^+ \delta^{\frac 52} >0,
\end{equation*}
for $\delta>0$ small enough. This transversality condition is enough to justify the existence of at least a point $\boldsymbol a_\sharp^+(0) \in 
\bar{\mathcal B}_{\RR^K}( \delta^\frac 54)$ such that $\TS=\infty$.

Indeed, for the sake of contradiction assume that for all $\boldsymbol a^+(0) \in \bar{\mathcal B}_{\RR^K}( \delta^\frac 54)$, it holds $\TS<\infty$.
Then, a contradiction follows from the following observations (see for instance more details in \cite{CMMgn} or in~\cite[Section 3.1]{CMkg}).

\emph{Continuity of $\TS$.} The above transversality condition implies that the map 
\begin{equation*}
\boldsymbol a^+(0) \in \bar{\mathcal B}_{\RR^K}(\delta^\frac 54)
\mapsto \TS \in [0,\infty)
\end{equation*}
is continuous and that $\TS=0$ for $\boldsymbol a^+(0) \in \mathcal S_{\RR^K}(\delta^\frac 54)$.

\emph{Construction of a retraction.} As a consequence, the map
\begin{equation*}
\boldsymbol a^+(0) \in \bar{\mathcal B}_{\RR^K}( \delta^\frac 54)
\mapsto \boldsymbol a^+(\TS) \in \mathcal S_{\RR^K}( \delta^\frac 54)
\end{equation*}
is continuous and its restriction to the sphere $\mathcal S_{\RR^K}( \delta^\frac 54)$ is the identity.

This is a contradiction with the no retraction theorem for continuous maps from the ball to the sphere.

At this point, we have proved the existence of $\boldsymbol a_\sharp^+(0) \in 
\bar{\mathcal B}_{\RR^K}( \delta^\frac 54)$, associated with a global solution $\vec u_\sharp \in \mathcal C([0,\infty), H^1)$ of \eqref{nlkg}, which also satisfies \eqref{BS:2} for all $t \ge 0$ (and $F=F_-$, $F_+=0$). Applying the results of Section \ref{S:5} to $\vec u_\sharp$, we infer that $\mathcal N\lesssim t^{-1}$ (so that \eqref{eq:th:2} holds) and that there exists $y_\sharp \in \RR$ with, for $k=1, \dots, K$,
\[ y_k = \bar y_k + y_\sharp + O( t^{-\theta+1}) \]
As $z_k = y_k - \frac{\ell_k}{2\alpha}$ and $|\ell_k| \le \mathcal N \lesssim t^{-1}$, $u_\sharp$ has the requested properties and the proof is complete.
\end{proof}

\end{document}